\documentclass{article}

\usepackage{arxiv}

\usepackage[utf8]{inputenc} 
\usepackage[T1]{fontenc}    
\usepackage{hyperref}       
\usepackage{url}            
\usepackage{booktabs}       
\usepackage{amsfonts}       
\usepackage{nicefrac}       
\usepackage{microtype}      
\usepackage{lipsum}
\usepackage{graphicx}

\usepackage{mathrsfs}       
\usepackage{amsmath,amssymb}
\usepackage{amscd}          
\usepackage{amsthm}         
\usepackage{tocloft}        
\usepackage[section]{placeins} 
\usepackage{subcaption}
\usepackage{cleveref}

\usepackage{xargs}                      
\usepackage[pdftex,dvipsnames]{xcolor}  
\usepackage[colorinlistoftodos,prependcaption,textsize=tiny]{todonotes}
\newcommandx{\unsure}[2][1=]{\todo[linecolor=red,backgroundcolor=red!25,bordercolor=red,#1]{#2}}
\newcommandx{\change}[2][1=]{\todo[linecolor=blue,backgroundcolor=blue!25,bordercolor=blue,#1]{#2}}
\newcommandx{\info}[2][1=]{\todo[linecolor=OliveGreen,backgroundcolor=OliveGreen!25,bordercolor=OliveGreen,#1]{#2}}
\newcommandx{\improvement}[2][1=]{\todo[linecolor=Plum,backgroundcolor=Plum!25,bordercolor=Plum,#1]{#2}}
\newcommandx{\thiswillnotshow}[2][1=]{\todo[disable,#1]{#2}}

\usepackage{algorithmic}
\usepackage{algorithm}

\newcommand{\diff}{\mathrm{d}}

\newtheorem{theorem}{Theorem}[section]
\newtheorem{corollary}{Corollary}[theorem]

\newtheorem{definition}{Definition}[section]
\newtheorem*{remark}{Remark}

\graphicspath{ {./images/} }
\numberwithin{equation}{subsection}
\numberwithin{theorem}{section}

\newcommand{\tabref}[1]{{\tablename~\ref{#1}}}

\title{An Eulerian approach to regularized JKO scheme with low-rank tensor decompositions for Bayesian inversion}

\author{
 Vitalii Aksenov \\
  WIAS \\
  Berlin, Germany \\
  \texttt{vitalii.aksenov@wias-berlin.de} \\
   \And
 Martin Eigel \\
  WIAS \\
  Berlin, Germany \\
  \texttt{martin.eigel@wias-berlin.de} \\
}

\begin{document}
\maketitle
\begin{abstract}
    The possibility of using the Eulerian discretization for the problem of modelling high-dimensional distributions and sampling, is studied.
    The problem is posed as a minimization problem over the space of probability measures with respect to the Wasserstein distance and solved with entropy-regularized JKO scheme.
    Each proximal step can be formulated as a fixed-point equation and solved with accelerated methods, such as Anderson's.
    The usage of low-rank Tensor Train format allows to overcome the \emph{curse of dimensionality}, i.e. the exponential growth of degrees of freedom with dimension, inherent to Eulerian approaches.
    The resulting method requires only pointwise computations of the unnormalized posterior and is, in particular, gradient-free.
    Fixed Eulerian grid allows to employ a caching strategy, significally reducing the expensive evaluations of the posterior.
    When the Eulerian model of the target distribution is fitted, the passage back to the Lagrangian perspective can also be made, allowing to approximately sample from it.
    We test our method both for synthetic target distributions and particular Bayesian inverse problems and report comparable or better performance than the baseline Metropolis-Hastings MCMC with same amount of resources.
    Finally, the fitted model can be modified to facilitate the solution of certain associated problems, which we demonstrate by fitting an importance distribution for a particular quantity of interest.

    We release our code at \url{https://github.com/viviaxenov/rJKOtt}
    
\end{abstract}

\section{Introduction}

The present paper focus on acquiring a tractable model of some complicated probability distribution $\rho_\infty$ such that an efficient sampling process is achieved.
We assume that this distribution is absolutely continuous on some $X \subseteq \mathbb{R}^d$ and in the notation do not distinguish between the distribution and its probability density function.
We focus on the setting where the target distribution is accessible via pointwise evaluation of the unnormalized probability density function.
One of the most important practical tasks falling into this setting is the Bayesian inverse problem, which we consider as an application of the proposed method.
Oftentimes, parameters of a physical system need to be estimated, along with the quantification of the uncertainty of the estimation, from indirect measurements.
In the Bayesian setting, this is done by studying the Bayesian posterior distribution
\begin{equation*}
    \rho_\infty(x|y) \propto \rho_0(x) \ell(F(x), y),
\end{equation*}
where $x$ stands for the unknown parameters, the \emph{forward operator} $F$ describes the response of the system with given parameters $x$, $y$ is the noisy measurement of the said response, and $\rho_0$ and $l$ describe the belief on the prior distribution of the parameters and the distribution of noise. 
For example, if an inverse problem for a heat equation is considered, $x$ could be the parameters, describing the initial distribution of the temperature, and the operator $F$ would map these parameters to temperatures at a set of fixed temporal and spatial points, in which the  measurements were performed.

The evaluation of $F$ quite often requires computationally intense calls to numerical solvers for the underlying {PDE-} or ODE-based models.
In application with more complex measurement systems such as astrophysical ones, the computation of the likelihood function can also be an expensive subroutine~\cite{aghanim2020planck}.
Although either uniform or Gaussian priors are usually assumed, it can be advantageous to use more involved priors, e.g. derived from data in medical imaging problems~\cite{liu2024population}.
Thus, the computation of the posterior distribution easily becomes the bottleneck in the numerical algorithms for such problems.
Hence, the number of these evaluations required to produce an sufficiently accurate sampling algorithm is a suitable measure to compare the performance of different approaches.

Although in principle the formula above contains all information about the posterior, actual uncertainty quantification tasks such as evaluation of expectations and credible intervals of quantities of interest requires numerical integration.
In general, its complexity scales exponentially with the dimension $d$ of the problem (i.e. the number of model parameters to be determined), rendering this approach impractical even for moderate dimensions.

\subsection{Sampling approaches}
    One of the possible approaches is to produce an approximate sample from the target, i.e. $\{x_i \}_{i=1}^N:\ x_i\sim \rho_\infty \text{i.i.d.}$.
    A large class of sampling methods, called \emph{Markov Chain Monte-Carlo}, iterate through a Markov Chain, which has the target as its invariant distribution. 
    In the Metropolis-Hastings algorithm, a new iterate is proposed based on the previous one, and then the proposal is accepted or rejected, with the criterion being designed to satisfy the \emph{detailed balance} condition, which, in turn, guarantees that the unique stationary distribution is $\rho_\infty$.
    If the gradient information can be utilised, one can consider the \emph{Langevin dynamics} method~\cite{roberts1996exponential}.
    Noticing that the  stochastic differential equation (SDE)
    \begin{equation}\label{eq:langevin_sde}
        \mathrm{d}X_t = \frac12\nabla\log\rho_\infty(X_t)\mathrm{d}t + \mathrm{d}W_t
    \end{equation}    
    with $W_t$ is the standard Brownian motion, has $\rho_\infty$ as a stationary distribution and converges in probability and its law converges in distribution to it, one can consider a discretization of this equation, for example, with Euler-Maruyama method.
    Since the time discretization converges to the target distribution only when time step converges to zero, it can be combined with a Metropolis-Hastings rejection step, yielding the \emph{Metropolis-adjusted Langevin Algorithm} (MALA).
    It is a good practice to run multiple chains in parallel, for example, to reduce variance in the estimation of the autocorrelation time of the chain or track convergence with Gelman-Rubin criterion~\cite{gelman1992single}.
    The next step in this direction is to consider the iterates in the parallel chains as an ensemble of particles, and use the information from the whole ensemble to further improve the properties of the algorithm. 
    In~\cite{goodman2010ensemble}, a set of proposal distributions, defined on the whole ensemble, is proposed, which ensures the affine invariance property of the algorithm. 
    In practice, this property means that convergence rate is invariant under the affine transformation in the state space.
    Combining the idea of the particle ensemble and gradient information yields the \emph{Affine-Invariant Langevin Dynamic} (ALDI) sampler~\cite{garbuno2020affine}, which is a time discretization of the SDE
    \begin{equation}\label{eq:aldi_sde}
        \mathrm{d}X^{(i)}_t = C(X_t)\nabla_{X^{(i)}_t}\log \rho_\infty(X^{(i)}_t)\mathrm{d}t
            + \frac{d + 1}{N}\left( X^{(i)}_t - m(X_t)\right)\mathrm{d}t
            + \sqrt{2}C^{\frac{1}{2}}(X_t)\mathrm{d}W^{(i)}_t
    \end{equation}
    with $X_t = \left(X^{(1}_t, \dots, X^{(N)}_t \right)$ being the ensemble of $N$ particles, and $m(X_t),\ C(X_t)$ are respectively the empirical mean and covariance of the ensemble. 
    The efficiency and robustness of this method can be further improved with such numeric techniques as ensemble enrichment and homotopy~\cite{eigel2022less}.

    The number of posterior calls in a Metropolis-Hastings MCMC can be  estimated as $O(N_{\text{chains}}\cdot N_{\text{steps}})$.
    Usually, a (possibly substantial) fraction of the initial steps, referred to as \emph{burn-in stage}, is discarded since the distribution of the chains first has to converge to the invariant distribution to be admissible.
    The consequent draws are not independent with the chain requiring multiple steps to <<forget>> the influence of the predecessors, which can be estimated by the \emph{autocorrelation time} $\tau_{ac}$.
    Because of this, only a $\frac{1}{\tau_{ac}}$ fraction of the performed steps yields independent samples.
    The computation of the gradient, e.g. required by MALA and ALDI, in principle can be achieved in $O(1)$ time compared to the computation time of the posterior if the numerical implementation of $\rho_\infty$ supports \emph{fast automatic differentiation}~\cite{evtushenko1998computation}.
    Practically, this is not the case if a rather complicated numerical solver for the forward model is used.
    The gradient hence might have to be estimated via finite-difference formulas, adding an $O(d)$ factor to the complexity.
    All these considerations render MCMC approaches time-consuming.
    Additionally, the extensive amount of the posterior calls cannot be reused in any way.
    
\subsection{Gradient flow structure of sampling approaches}
    The seminal paper of Jordan, Kinderlehrer and Otto~\cite{jordan1998variational}, the Langevin process \eqref{eq:langevin_sde} is viewed from the perspective of the geometry of the space of probability distribution endowed with the Wasserstein distance.
    To be more precise, the process is proved to be the gradient flow of the \emph{Kullback-Leibler divergence}
    \begin{equation}\label{eq:KL_definition}
        \operatorname{KL}(\rho|\rho_\infty) = \int \log\left( \frac{\rho}{\rho_\infty}\right)\mathrm{d}\rho,
    \end{equation}
    which qualitatively means that the distribution is gradually changing in the direction of the steepest descent of the functional.
    In general, a gradient flow of a functional $\mathcal{E}$ starting from some initial distribution $\rho^0$ can be defined as a continuous time limit (in case one exists) of the proximal scheme
\begin{equation}\label{eq:JKO_general_definition}
        \rho^{k+1} = \arg\min_\rho \mathcal{E}(\rho) + \frac{1}{2T_k} W^2_2(\rho, \rho^k), \quad k = 0, 1, \dots 
    \end{equation}
    as $\max_k T_k \to 0$.
    The scheme is referred to as \emph{minimizing movements} (MM) or, in case that $\mathcal{E} = \operatorname{KL}(\cdot | \rho_\infty)$, the \emph{Jordan, Kinderlehrer and Otto} (JKO) scheme.
    The mathematical background regarding existence, uniqueness and rates of convergence to the optimum for the gradient flows is comprehensively presented in~\cite{ambrosio2005gradient}.
    The theory provides a unified framework for mathematical modelling of a multitude of natural processes, which are characterized by conservation of mass and monotone decrease of a certain functional.
    Evolutions governed by diffusion, porous media~\cite{otto2001geometry}, quantum drift-diffusion~\cite{gianazza2009wasserstein} or Keller-Segel~\cite{blanchet2013gradient} equations all exhibit the structure of a gradient flow with respect to the Wasserstein distance.
    The approach can be extended to models with a discontinuous functional, where the PDE model does not exist, e.g. crowd motion with congestion~\cite{leclerc2019lagrangian}.
    
    From a numerical perspective, using the Wasserstein framework allows to construct novel algorithms for the evolution of distributions based on the various reformulations and regularization of the optimization problem~\eqref{eq:JKO_general_definition}.
    For instance in~\cite{benamou2016augmented,carrillo2022primal} the primal-dual reformulation is studied.  
    Another important variation is based on the entropic regularization of the problem.
    With this, the respective entropic JKO scheme is studied in~\cite{carlier2017convergence}.
    It is shown that under a certain coupling law between the time step and the amount of the entropic regularization, the limit of the entropic scheme is the original gradient flow.
    Numerical examples for this scheme are presented, for example, in~\cite{peyre2015entropic}.
    The approach presented relies on using the proximal stepping in the Wasserstein directly, instead of considering the PDE model of its continuous time limit, the latter either not existing, or being excessively complex. 
    This idea is demonstrated by tackling such problems as crowd motion with congestion and a gradient flow on a Riemannian manifold.
    
    In cases when gradient flows are viewed in the context of modelling natural processes, the dimension of the problem is typically not higher than $d = 3$ and Eulerian approaches with spatial discretization of densities and other appropriate functions are frequently used.
    There are examples of finite difference~\cite{bailo2020}, finite volume~\cite{cances2020variational} and finite element~\cite{benamou2016augmented} discretizations.
    However, as the amount of degrees of freedom required for the spatial discretization normally scales exponentially with the dimension, Lagrangian techniques are preferred in higher dimensions.
    These approaches consider either the behaviour of a particle ensemble or parametrize a flow transporting particles in such a way that they have the target distribution.
    A notable example is \emph{Stein variational gradient descent} (SVGD) introduced in~\cite{liu2016stein} and analyzed in the framework of gradient flows in~\cite{liu2017stein}.
    In this method one assumes the intermediate transports appearing in the gradient flow to lie in a certain \emph{reproducing kernel Hilbert space} with which a tractable deterministic update formula to compute the evolution of the ensemble is obtained. 
    The approach proves to be quite powerful and can for instance be extended to a gradient-free setting~\cite{han2018stein} or to gradient flows with respect to different distances~\cite{maurais2024sampling}.

\subsection{Low-rank Tensor decompositions}
    A central motivation for this work is that -- in our opinion -- the explicit discretization as in Eulerian methods is a more structured and organized way to model posterior distributions than using Langrangian approaches as mentioned above.
    For the latter, the position of the next particle is not known in advance and each step requires the renewed computation of the movement toward the posterior distribution.
    The new position might in fact be close to one of the previously encountered positions and thus not bear a lot of additional information to improve the solution.
    In the Eulerian approach, the posterior is evaluated on a prescribed grid.
    We hypothesize that if the posterior has some additional structure, the values on the entire grid can be approximated within reasonable tolerance by computing only the values at a few nodes. 
    Additionally, since the nodes are known in advance, a caching strategy can be implemented, rendering the posterior calls reusable.
    Additionally, the Eulerian discretization can be adapted to the particular problem at hand and the required precision of the solution.
    
    The goal of the current work is to explore how Eulerian approaches can be efficiently adapted for higher dimensions ($d \geq 4$).
    To alleviate the exponential growth of the number of degrees of freedom in terms of the dimension, a compression format for multidimensional arrays (\emph{tensors}) is required.
    One has to represent the initial data in the compressed format and be able to perform all required operations of the algorithm exactly or approximately without leaving the format.
    In the current work we focus of the popular \emph{Tensor-Train} (TT) decomposition~\cite{oseledets2010tt}.
    A tensor $A \in \mathbb{R}^{N_1 \times \dots \times N_d}$ is in the TT format if it has the form
    \begin{equation}\label{eq:tt_definition_intro}
        A_{i_1\dots i_d} = A^1_{i_1 j_1}A^2_{j_1 i_2 j_2}\cdot\dots\cdot A^d_{j_{d-1}i_d},
    \end{equation}
    where for each $k \in \overline{1,\ d}$, $A^k_{j_{k-1}i_k j_k} \in \mathbb{R}^{r_{k-1} \times N_k \times r_k}$ is a $3$-dimensional tensor, $r_0 = r_{k} = 1$ and Einstein summation convention is assumed.
    Values $r_k$ are called \emph{TT-ranks}.
    If $\max_k r_k = r,\ \max_k N_k = N$, the storage complexity of the tensor is only $\mathcal{O}(dNr^2)$, which in case of low-rank $r$ can be drastically less than $\mathcal{O}(N^d)$ required for the storage of the uncompressed full tensor.
    
    For a given tensor in format \eqref{eq:tt_definition_intro}, its approximation with lower rank and controlled error can be computed by the TT-SVD algorithm.
    Linear algebra operations such as linear combinations, index contractions, inner products or norms can also be efficiently computed.
    If the TT decomposition is used for the discretization of a function, index contraction corresponds to integrating with respect to a subset of variables.
    This means that some typically complex tasks such as the computation of marginal distributions become tractable.
    A plethora of approximate methods have been adapted for the TT format.
    For example, a general minimization problem with respect to an argument in TT format can be solved with the \emph{alternating linear scheme} (ALS) or Modified ALS (MALS), also known as \emph{density matrix renormalization group} (DMRG)~\cite{holtz2012alternating}.
    In the setting where one strives to obtain a low-rank TT approximation of a tensor, each individual entry can be computed by \emph{tensor completion}, e.g. by using a \emph{cross approximation} (TT-cross)~\cite{oseledets2010tt}.
    
    With various algorithms at hand, the TT format has been applied to a variety of practical problems.
    Of relevance to the current topic would be, firstly, density estimation where TT are either used directly~\cite{novikov2021tensor} or in conjunction with Normalizing Flows (NF)~\cite{ren2023high}.
    The approximation of distributions, accessible only via pointwise evaluations of unnormalized densities, was studied in~\cite{cui2022deep}.
    In this paper, the strategy is to construct a series of intermediate densities. A transport between them is approximated via the \emph{inverse Rosenblatt transport}.
    Additionally, tensor trains are reported to be used for PDEs in genereal (see review~\cite{bachmayr2023low}), and, in particular, parabolic PDEs, similar to the ones appearing in the current work ~\cite{richter2021solving, chertkov2021solution}.

\subsection{Contribution and paper organization}
    In this paper, we present a novel method for approximating probability distributions. It aims at applications in Bayesian inversion~\cite{stuart2010inverse}.
    The underlying idea is to minimize a functional $\mathcal{E} = \operatorname{KL}(\cdot | \rho_\infty)$, which quantifies the dissimilarity to the posterior distribution.
    The entropic regularization of the scheme~\eqref{eq:JKO_general_definition} is used to produce a discrete sequence of successive approximations of the target distribution in a tractable compressed representation.
    The optimality condition for the proximal step is reformulated in terms of entropic potentials as a system of nonlinearily coupled heat equations.
    The system is explicitly discretized in an Eulerian way. 
    Subsequently, the heat equations are solved by means of finite difference methods on a high-dimensional grid.
    The coupling is then considered a fixed-point problem, which is tackled by accelerated fixed-point methods.
    To the authors' knowledge, this is the first use of fixed-point methods together with the low-rank tensor-train format. 
    As it will be revealed below, the underlying dynamical formulation of the proximal steps provides dynamics that interpolate between the approximate measures in the form of an ODE or an SDE.
    Numeric solution of these equations provides a way to efficiently draw samples from the model distributions.

    The presented method has the same inputs as the common sampling method such as Metropolis-Hastings MCMC or SVGD, namely the unnormalized posterior density and general a priori parameter bounds.
    We implement the caching of the posterior calls during the solution process. In the Bayesian setting, these contain the expensive solution operator of the model. Together with the accelerated fixed-point scheme, the caching noticeably increases the efficiency of the method.
    Once the model is fitted, certain computations can be performed without additional posterior calls.
    For example, a normalized approximation of the posterior or its marginal distributions are accessible.
    This is important for Bayesian inversion since any statistical quantification of the determined parameters can be carried out efficiently.
    The approximation of the dynamics, which evolves the particles to the target distribution, is also possible.
    This means that an arbitrary amount of the samples from the approximate posterior can be drawn without additional (costly) posterior calls. 
    Finally, the acquired model can be modified to facilitate the solution of a sampling problem for a distribution, similar to the Bayesian posterior.
    We demonstrate it by approximating the optimal importance distribution for a certain quantity of interest.
    
    Previous research closest to the present one can be found in~\cite{chertkov2021solution} and~\cite{han2024tensor}.
    In~\cite{chertkov2021solution}, the Fokker-Planck equation with a general, possibly time-dependent drift term is considered.
    The solution is determined by splitting the operator, describing either the drift or diffusion process.
    The diffusion term is quite well suited for approximation in the TT format due to its intrinsic low-rank structure.
    We utilize a similar approach to solve the heat equation that arises in our formulation.
    The solution of the drift part can be obtained pointwise by solving the characteristic ODEs.
    The TT representation for the whole solution is obtained by means of a tensor reconstruction with the TT-cross algorithm.
    Formally, this method can be applied for the approximation of a target density by setting the drift term $f$ equal to the score of the distribution: $f(x, t) = \nabla\log\rho_\infty(x)$.
    However, the implementation would require the gradient of the posterior, which to be avoided for the sake of computational efficacy.
    Although quite similar to our method in the range of applied numerical techniques, the other method comes from a different context for simulating an actual physical system.
    Thus, it is not clear if the method would be numerically efficient in a setting where only approaching the invariant distribution is of importance.
    In~\cite{han2024tensor}, the same entropy-regularized proximal steps are applied to a linearized version of the target functional.
    The authors argue that under a certain choice of coupling between the time step and the regularization factor, the limit of the regularized scheme for the linearized functional is again the desired distribution.
    With a TT-cross approximation of a certain intermediate quantity, an efficient deterministic update formula is acquired and a particle ensemble can be evolved towards the target in a fashion somewhat similar to SVGD.
    We demonstrate that the regularized step for the original method, despite requiring the solution of a nonlinear coupling problem, can still be tackled with appropriate tools. 
    Moreover, we argue that the direct approximation of the potentials and densities opens up possibilities to extension of the method.
    An example is importance sampling for the expectations of quantities of interest with respect to the posterior. 
    Additionally, the two aforementioned papers report experiments in dimensions up to $d = 6$.
    While it is already impossible to tackle the considered problem numerically without any compression such as tensor trains, in our opinion this dimension is still at the lower end for reasonable practical applications.
    Thus, we verify the performance of our method in much higher dimension up to $d = 30$.
    
    The remaining paper is structured as follows.
    In the second section, reference information on Wasserstein spaces is provided.
    In the third section, the PDE system underlying the method is derived.
    The fourth section contains details of the implementation.
    The subsequent two sections contain the numerical experiments.
    Firstly, the method is verified on trial distributions.
    We then present the use cases of a Bayesian inverse problem with inference of an initial condition with a parabolic and a hyperbolic PDE.
    The manuscript is concluded with a short outlook on possible further developments.

\section{Metric space of measures with respect to the Wasserstein distance}    
    The overall idea of the current approach is to treat the task of sampling from the target distribution or to approximate it with a tractable model as a minimization of a certain functional, that quantifies the dissimilarity to the target.
    We denote the functional $\mathcal{E}$ and in what follows fix it to be the $\operatorname{KL}$ divergence:
    \[
        \mathcal{E}(\rho) = \operatorname{KL}(\rho | \rho_\infty) = \int\log\left( \frac{\rho}{\rho_\infty}\right)\mathrm{d}\rho.
    \]
    This functional is of particular importance because the unnormalized density is used in its computation instead of the true one.
    Since the value of the functional would only change by a constant the minimizer stays unchanged.
    In the sequel, however, we would use the $\mathcal{E}$ notation for the functional, except for explicitly mentioned parts, to signify that the proximal method can be used for a variety of functionals (as long as the aforementioned parts can be efficiently tackled numerically).

    Probability measures lack a natural linear space structure.
    However, if a metric structure is imposed, an analog of a steepest descent can be defined, either as a discrete sequence of proximal steps (minimizing movements) or as a continuous process (a \emph{gradient flow}).
    The Wasserstein distance is based on the solution of an optimal transport problem.
    Unlike other metrics defined on probability measures such as Hellinger or total variance (TV), it is takes into account the metric of the underlyig space, thus being <<aware>> of its the topology.
    Additionally, transport maps that are optimal with respect to this distance can be shown to exist under certain conditions.
    This motivates numerous applications of the Wasserstein distance, both as an element of the analysis in PDE theory (e.g.~\cite{jordan1998variational,otto2001geometry,liero2013gradient,craig2017nonconvex} and references therein) and as a practical tool in statistics~\cite{bernton2019statistical}, machine learning~\cite{torres2021survey}, image processing~\cite{bonneel2023survey} and others.
    \cite{ambrosio2005gradient} provides a detailed reference on the gradient flows both in general metric spaces and in the Wasserstein space.   
    In the rest of the section we recall the concepts required for the definition of our method.

\subsection{Wasserstein distance and its dynamic reformulation}
    In the most generality, the Wasserstein distance can be defined on any separable metric space $X$ with the Radon property given by
    \begin{equation}\label{eq:wasserstein_distance}
        W^p_p(\rho_0, \rho_1) =  \min\limits_{\pi \in \Pi(\rho_0, \rho_1)} \int\limits_{X \times X}d^p(x, y)\mathrm{d}\pi(x, y)
    \end{equation}
    with Borel probability measures $\rho_0,\ \rho_1$ on $X$, the set of probability measures on $X\times X$ with marginals $\rho_0$ and $\rho_1$ denoted by $\Pi(\mu, \nu)$ metric $d$ on $X$.
    The set of probability measures with finite $p$-moments defined by
    \begin{equation}
        \mathscr{P}_p(X) = \left \{ \rho \in \mathscr{P}(X): \int\limits_X d^p(x, x^0) \diff \rho(x) \right \}
    \end{equation}
    is a separable metric space with respect to the $p$-Wasserstein distance.
    It is complete if $X$ is complete.
    Henceforth we consider $X = \mathbb{R}^d$ with Euclidean metric $d(x,y) = \|x - y\|_2^2$ and the $2-$Wasserstein.
    
    One can study a notion of curves in the Wasserstein space to define the dynamics, that smoothly interpolates between the initial and terminal distributions.
    \begin{definition}{Absolutely continuous curve}
        A curve $\rho(t): [a, b] \mapsto \mathscr{P}_2(X)$ is called \emph{absolutely continuous} if $\exists m \in L^1([a,b])$ such that
        \[
            W_2^2(\rho(s), \rho(t)) \leq \int \limits_s^t m(r)dr,\quad \forall a \leq s < t \leq b.
       \]
    \end{definition}
    In $\mathscr{P}_2$ there is an important characterization of the absolutely continuous curves with the continuity equation.
    \begin{theorem}{Continuity equation~\cite[Theorem~8.3.1]{ambrosio2005gradient}}
        There exists a Borel vector field $v_t \in L^2(\rho_t, X):\ \|v_t\|_{L^2(\rho_t, X)} \leq |\rho_t'|$ and the continuity equation
        \[
            \partial_t \rho_t + \nabla(\rho_t v_t) = 0
        \]
        holds in the sense of distributions, i.e.
        \[
            \int\limits_{t_1}^{t_2}\int\limits_{X} \left(
                                                    \partial_t \varphi(x,t) + \left\langle v_t(x), \nabla_x \varphi (x,t) \right\rangle
                                                    \right)\diff \rho_t \diff t = 0,\quad \varphi \in \operatorname{Cyl}(X\times [t_1, t_2]).
        \]
    \end{theorem}
    Under additional regularity assumptions on the vector field $v_t$ (cf.~\cite{ambrosio2006stability} or \cite[{Lemma {8.1.6}}]{ambrosio2005gradient}), there exists a diffeomorphism $X_t$ such that the solution admits the following characterization:
    \begin{gather}\label{eq:ode_characteristic_method}
        \rho_t = (X_t)_\sharp \rho_0, \notag \\
        \frac{\diff}{\diff t}X_t(x, t) = v_t(X_t(x, t), t),  \label{eq:diffeomorphism} \\
        X_t(x, 0) = x \notag
    \end{gather}
    In other words, a passage to the Lagrangian point of view can be made, in which the dynamics of individual particles is described.
    From the numerical point of view this means that given samples from $\rho_0$ a numerical model (approximating the interpolation of measures in the Eulerian sense) can be used to provide samples from the distribution $\rho_1$.
    
    Based on the interpretation of absolutely continuous curves via the continuity equation, one can reformulate the optimal transport  problem as a problem of finding the geodesic curve in the Wasserstein space.
    This leads to the \emph{dynamic optimal transport} problem due to Benamou and Brenier~\cite{benamou2000computational}:
    \begin{align}\label{eq:dynamic_OT}
        W^2_2(\rho_0, \rho_1) = &\min\limits_{v_t}\int\limits_0^1\int\limits_X \|v_t(x)\|^2\rho_t(x)\diff{x}\diff{t}, \\
                              &\operatorname{s.t. }   
                              \begin{aligned}\notag
                              \partial_t \rho_t + \nabla(\rho_t v_t) &= 0,\\
                                                     \rho_t(0) &= \rho_0 \\ 
                                                     \rho_t(1) &= \rho_1,
                              \end{aligned}
    \end{align}
    complemented by a no-outflow boundary condition in case $X$ has a boundary.

    \subsection{Minimizing movements scheme}
    The minimization of the target functional $\mathcal{E}$ can be performed iteratively via proximal steps,
    \begin{equation}\label{eq:JKO_scheme}
        \rho^{k+1} = \arg\min\limits_{\rho} \mathcal{E}(\rho) + \frac{1}{2T_k}W^2_2(\rho, \rho^k).
    \end{equation}
    This is the counterpart of the implicit gradient descent in the Wasserstein space. 
    If one is explicitly interested in the curve interpolating between successive steps is of interest, the dynamic formulation of the optimal transport problem can be used with~\eqref{eq:JKO_scheme}.
    One then arrives at the PDE-constrained minimization problem
    \begin{align}\label{eq:dynamic_JKO}        
        \rho^{k+1}=\arg&\min\limits_{\rho_t,v_t}\int\limits_0^1\int\limits_X \|v_t(x)\|^2\rho_t(x)\diff{x}\diff{t} + 2T_k \mathcal{E}(\rho_1), \\
                              &\operatorname{s.t.}\quad   
                              \begin{aligned}\notag
                              \partial_t \rho_t + \nabla(\rho_t v_t) &= 0,\\
                                                     \rho_t(0) &= \rho^k. 
                              \end{aligned}
    \end{align}

    Just as in the case of minimization over a linear space, convexity of the target functional allows to determine the convergence rate of the method.
    Unfortunately, it is known that $W^2_2$ is not convex along geodesics, which complicates the analysis. 
    Thus, a generalization is required leading to the notion of \emph{convexity along generalized geodesics}.
    \begin{definition}{Convexity along generalized geodesics~\cite[Definition~{9.2.4}]{ambrosio2005gradient}}        
        The functional $\mathcal{E}:\ \mathscr{P}_2(X) \to \mathbb{R}$ is said to be \emph{$\lambda$-convex along generalized geodesics} for some $\lambda \in \mathbb{R}$ if for any $\rho_1,\ \rho_2,\ \rho_3 \in D(\mathcal{E})$ there exists a plan 
        \begin{equation*}
            \boldsymbol{\rho} \in \Pi(\rho_1, \rho_2, \rho_3): p^{1,2}_\sharp\boldsymbol{\rho} \in \Pi_{\text{opt}}(\rho_1, \rho_2),\ p^{1,3}_\sharp\boldsymbol{\rho} \in \Pi_{\text{opt}}(\rho_1, \rho_3)
        \end{equation*}
        and a generalized geodesic, which is a curve of the form 
        \begin{equation*}
            \mu^{2\to 3}_t = \left(t\pi^2 + (1 - t)\pi^3 \right)_\sharp\boldsymbol{\rho} 
        \end{equation*}
        such that
        \begin{equation}\label{eq:agg_convexity}
            \mathcal{E}(\rho_t^{2\to3}) \leq (1 - t)\mathcal{E}(\rho_2) 
                    + t\mathcal{E}(\rho_3)
                    - \frac{1}{2}\lambda t(1 - t) W^2_{\boldsymbol{\rho}}(\rho_2, \rho_3).
        \end{equation}
        Here, $D$ is the domain of the functional, $\Pi(\rho_1, \rho_2, \rho_3)$ are probability distributions over $X^3$ with one-dimensional marginals $\rho_i$, $i\in\{1,2,3\}$, $\Pi_\text{opt}(\rho_i, \rho_j)$ is the set of optimal plans between $\rho_i$ and $\rho_j$ (in the sense of \eqref{eq:wasserstein_distance}), $p^i$ is the projection onto the $i$-th variable, i.e. $p^i(x_1, \dots, x_K) = x_i$, and $W^2_{\boldsymbol{\rho}}(\rho_2, \rho_3)$ is defined as
        \begin{equation*}
            W^2_{\boldsymbol{\rho}}(\rho_2, \rho_3) := \int_{X^3}|x_2 - x_3|^2 d\boldsymbol{\rho}(x_1, x_2, x_2) \geq W^2_2(\rho_2, \rho_3).
        \end{equation*}
    \end{definition}
    The convexity of $\operatorname{KL}(\cdot | \rho_\infty)$ depends on the properties of the target distribution.
    A characterization can be obtained as a straightforward corollary of Propositions~{9.3.2} and~{9.3.9} of \cite{ambrosio2005gradient}:
    \begin{corollary}
        Let $X = \mathbb{R}^d$.
        Let the target density have the form $\rho_\infty = e^{-V}\operatorname{d}\mathcal{L}^d$ with the Lebesgue measure ${L}^d$ and $V: \mathbb{R}^d \to \mathbb{R}$ is $\lambda$-strongly convex for some $\lambda \geq 0$.
        Then the functional $\operatorname{KL}(\cdot | \rho_\infty)$ is $\lambda$-convex along generalized geodesics.
    \end{corollary}
    \begin{proof}
        The $\operatorname{KL}(\cdot | \rho_\infty)$ functional decomposes into two terms:
        \begin{equation*}
            \operatorname{KL}(\rho|\rho_\infty) = \int \log\left( \frac{\rho}{\rho_\infty}\right)\mathrm{d}\rho =
                \int \log{\rho}\mathrm{d}\rho + \int V\mathrm{d}\rho
                = \mathcal{E}_\text{int}(\rho) + \mathcal{E}_\text{pot}(\rho).
        \end{equation*}
        Due to Proposition~{9.3.9}, the internal energy term $\mathcal{E}_\text{int}$ is geodesically convex ($\lambda$ = 0).
        Due to Proposition~{9.3.2}, the potential energy $\mathcal{E}_\text{pot}$ is $\lambda$-convex along any interpolating curve.
        Thus, the inequalities of type \eqref{eq:agg_convexity} can be summed, leading to $\lambda$-convexity of $\operatorname{KL}(\cdot | \rho_\infty)$.
    \end{proof}
    If the target functional is $\lambda$-convex along generalized geodesics, the sequence generated by the Minimizing Movement scheme can be characterized by \emph{Evolution Variational Inequalities}.
    \begin{theorem}{Variational inequalities, see \cite[{Theorem~{4.1.2}}]{ambrosio2005gradient}}
        Let $\mathcal{E}$ be $\lambda$-convex along generalized geodesics.
        For any $\rho \in \overline{D(\mathcal{E})},\ T: \lambda T > -1 $
        \begin{itemize}
            \item the proximal minimization problem \eqref{eq:JKO_general_definition} has a unique solution
                \begin{equation*}
                    \rho_{T} = \arg\min_{\tilde\rho} \frac{1}{2T}W^2_2(\tilde\rho, \rho) + \mathcal{E}(\tilde\rho) 
                \end{equation*}
            \item for any other $\tilde\rho \in D(\mathcal{E})$,
                \begin{equation}\label{eq:EVI}
                    \frac{1}{2T}W^2_2(\rho_T, \tilde\rho) - \frac{1}{2T}W^2_2(\rho, \tilde\rho) + \frac{1}{2}\lambda W^2_2(\rho_T, \tilde\rho) \leq \mathcal{E}(\tilde\rho) - \mathcal{E}_T(\rho),
                \end{equation}
            where 
                \begin{equation*}
                    \mathcal{E}_T(\rho) = \frac{1}{2T}W^2_2(\rho_T, \rho) + \mathcal{E}(\rho_T) \geq \mathcal{E}(\rho_T).
                \end{equation*}
        \end{itemize}
    \end{theorem}
    \begin{corollary}
        Let $\mathcal{E} = \operatorname{KL}(\cdot | \rho_\infty)$ be $\lambda$-convex along generalized geodesics.
        Then $\rho_\infty$ is a global minimum of the functional.
        For any $\rho$ and denoting $\tilde\rho = \rho_\infty$, one can see that
        \begin{equation*}
            W_2(\rho_T, \rho_\infty) \leq \frac{1}{\sqrt{1 + \lambda T}}W_2(\rho, \rho_\infty).
        \end{equation*}
        For the sequence generated by the minimizing movement scheme $\rho^{k+1} = \left(\rho^{k}\right)_{T_k}$ a convergence rate can be estimated by
        \begin{equation*}
            W_2(\rho^N, \rho_\infty) \leq \prod\limits_{k=1}^N \frac{1}{\sqrt{1 + \lambda T_k}}W_2(\rho^0, \rho_\infty).
        \end{equation*}
    \end{corollary}
    We have to admit that the standard analysis presented above is rather incomplete because, for example, the convexity property for the $\operatorname{KL}$ does not hold in case of a multimodal target measure.
    One also has to take the suboptimality of the solution into account.

\section{Entropic regularization of the Minimizing Movements scheme}
    To obtain a numerical solution it is quite common to regularize the computation of the Wasserstein distance (\eqref{eq:wasserstein_distance}) by adding a term, which makes the problem strongly convex.
    Although other approaches for example based on quadratic regularization~\cite{lorenz2021quadratically} and Rényi divergences~\cite{bresch2024interpolating} exist, the most prominent approach is \emph{entropic optimal transport}~\cite{nutz2022entropic}.
    At the cost of introducing a small bias controlled by the regularization parameter, one gets a strongly convex problem.
    When $\rho_{0,1}$ are two empirical measures (e.g. in a setting of comparing two point clouds as is popular in data science and machine learning), this enjoys better asymptotic complexity than the linear program solution required for the unregularized problem~\cite{dvurechensky2018computational}.
    Additionally, the regularized problem allows for efficient parallel implementation with the Sinkhorn algorithm~\cite{cuturi2013sinkhorn}.

    Inspired by the previous works~\cite{li2020fisher,li2023kernel} we introduce the regularization into the dynamic reformulation of the problem (\ref{eq:dynamic_JKO}),
    \begin{align}\label{eq:reg_dyn_JKO}
        \rho^{k+1} =  &\arg\min\limits_{v,\rho }\int\limits_0^{T_k}\int\limits_X {\frac{1}{2}\|v(t,x)\|^2}{\rho(t, x)}\diff{x}\diff{t} + \mathcal{E}(\rho({T_k},\cdot)), \\
                &\operatorname{s.t.}\quad   
                \begin{aligned}\notag
                    \partial_t \rho + \nabla(\rho v) &= \beta_k\operatorname{\Delta}\rho,\\
                    \rho(0, x) &= \rho^{k}(x).
                \end{aligned}
    \end{align}
    Essentially a diffusion term is added to the flow constraint.
    The parameter $\beta_k > 0$ controls the strength of the regularization.
    In \cite{li2020fisher} it is shown that this is in fact equivalent to adding entropic and Fisher information regularizing terms to the target functional.
    The equivalence between the dynamic and static regularized OT formulations is explained in~\cite{chen2016relation}, which can be obviously extended to the Minimizing Movements scheme.
    As will be evident from the derivation below, this particular form of regularization allows to reformulate the problem as a coupled PDE system.
    Using a fixed-point method is an obvious approach for the solution, similar to the well-known Sinkhorn algorithm and well suited for the implementation with Tensor Train decompositions.
    
    \subsection{Derivation of the method}
    
    We introduce a dual variable (as Lagrange multiplier) for the continuity equation constraint $\Phi(x, t)$.
    Writing down the Lagrangian and the optimality conditions one can see that $v = \nabla\Phi$ and it holds the system
    \begin{align}
        &\partial_t \rho + \nabla(\rho\nabla\Phi) = \beta_k\Delta\rho \label{eq:RegWass_proximal_rho},\\
        &\partial_t \Phi + \frac{1}{2}\|\nabla \Phi\|^2 = -\beta_k\Delta\Phi \label{eq:RegWass_proximal_Phi},\\
        &\rho(0, x) = \rho^k(x),\quad \Phi(T_k, x) = -\delta\mathcal{E}(\rho(T_k,\cdot), x). \label{eq:RegWass_proximal_end}
    \end{align}
    Here, $\delta \mathcal{E}$ is the first variation of the functional $\mathcal{E}$.
    This can be seen as a primal-dual formulation of the problem.
    For a general energy we see that the terminal condition has nonlinear dependence on $\Phi,\ \rho$.

   With a change of variables in terms of the (\emph{Hopf-Cole transform})
    \begin{equation}\label{eq:Hopf-Cole}
        \eta(t,x) = e^{\frac{\Phi(t,x)}{2\beta_k}},\quad 
        \hat{\eta}(t,x) = \rho(t,x)e^{\frac{-\Phi(t,x)}{2\beta_k}} 
    \end{equation}
    yielding
    \begin{equation}\label{eq:Hopf-Cole_inverse}
        \Phi = 2\beta_k\log\eta,\quad \rho = \eta\hat{\eta},
    \end{equation}
    the system of a forward-in-time Fokker-Planck \eqref{eq:RegWass_proximal_rho} and a backward-in-time Hamilton-Jacobi \eqref{eq:RegWass_proximal_Phi} transforms into a system of two heat equations, one forward and one backward in time.
    It is given by
    \begin{align}\label{eq:heat_system_general}
        &\partial_t \hat{\eta} = \beta_k\Delta\hat{\eta} \\
        &\partial_t {\eta} = -\beta_k\Delta\eta \\
        &\hat{\eta}(0, x) = \frac{\rho^k(x)}{\eta(0, x)}\\ 
        &\eta(T_k,x) = e^{-\frac1{2\beta_k}{\delta\mathcal{E}(\eta(T_k, \cdot), \hat\eta(T_k, \cdot), x)}}. \label{eq:heat_system_general_end}
    \end{align}

    \subsection{Fixed-point formulation of the proximal step}\label{ss:fp_equations}
    We argue that system~\eqref{eq:heat_system_general}-\eqref{eq:heat_system_general_end} is more suitable for the numeric solution than~\eqref{eq:RegWass_proximal_rho}-\eqref{eq:RegWass_proximal_end}.
    Firstly, now two similar linear PDEs are considered.
    The solution of the general initial-value problem for the heat equation can be obtained by a convolution with the fundamental solution, namely
    \begin{gather}
        \hat\eta(T_k, \cdot) = H_d(\beta_k T_k) \hat\eta(0, \cdot) \label{eq:heat_oper} \\
        \eta(0, \cdot) = H_d(\beta_k T_k) \eta(T_k, \cdot),     \label{eq:heat_oper_end}
    \end{gather}
    where $H_d: L_1(\mathbb{R}^d) \times \mathbb{R}^{+} \to L_1(\mathbb{R^d})$ is the integral operator such that
    \begin{gather}
        \left(H_d(s)u\right)(x) = \int\limits_{\mathbb{R}^d} G_d(s, y - x) u(y) \label{eq:heat_operator_d} \\
        G_d(s, \xi) = \frac{1}{(4\pi s)^{\frac{d}{2}}} e^{-\frac{1}{4 s}\|\xi\|^2}.
    \end{gather}
    
    We suggest that the problem of satisfying the whole system together with the initial and the terminal conditions lends itself better for a reformulation as a fixed-point iteration.
    Consider the following steps for some  $\eta_m$:
    \begin{align}
        \eta_{m,0} &= H_d(\beta_k T_k) \eta_{m}, \label{eq:fp_general} \\
        \hat{\eta}_{m,0} &= \frac{\rho^k(x)}{\eta_{m,0}}, \label{eq:fp_initial} \\ 
        \hat{\eta}_{m} &= H_d(\beta_k T_k) \hat{\eta}_{m,0}, \label{eq:fp_heat_hat_eta} \\
        \tilde\eta_m &= e^{-\frac1{2\beta_k}{\delta\mathcal{E}(\tilde\eta_m, \hat\eta_m, \cdot)}}.  \label{eq:fp_general_end}
    \end{align}
    If $(\eta^*, \hat\eta^*)$ were the solution of \eqref{eq:heat_system_general}-\eqref{eq:heat_system_general_end} with $\eta_m = \eta^*(T_k, \cdot)$, then $\eta_{m, 0} = \eta^*(0, \cdot),\ \hat\eta_{m, 0} = \hat\eta^*(0,\cdot),\ \hat\eta_m = \hat\eta^*(T_k, \cdot)$ and $\eta_m$ is the fixed point of the operator $G$, implicitly defined by \eqref{eq:fp_general}-\eqref{eq:fp_general_end}
    \begin{equation*}
        \tilde\eta_m = G(\eta_m) = \eta_m.
    \end{equation*}
    We propose to reformulate the coupling problem as a classical Picard iteration $\eta_{m + 1} = \tilde\eta_m,\ m = 0,\ 1, \dots$, starting from some initial guess $\eta_0$ and computing successive iterates until $\eta_m$ and $\tilde\eta_m$ are sufficiently close in some metric.
    The condition \eqref{eq:fp_general_end} is itself a nonlinear equation since the variation of energy depends on $\tilde\eta_m$.
    For general energy functional this can be relaxed by taking the value of $\eta$ from the previous iteration, i.e.,
    \begin{equation*}
        \tilde\eta_m = e^{-\frac1{2\beta_k}{\delta\mathcal{E}(\eta_m, \hat\eta_m, x)}}.
    \end{equation*}
    However, in the case of $\operatorname{KL}$ divergence as considered in the current work, one can see that
    \begin{equation*}
        \delta\operatorname{KL}(\rho|\rho_\infty)(x) = \log{\rho} - \log{\rho_\infty}  + 1 + \text{const.}.
    \end{equation*}
    Here, the constant appears because the target density $\rho_\infty$ is only known up to a multiplicative constant. 
    We note that adding a constant to the target functional does not change the $\arg\min$ in the minimization problem. 
    Thus we assume that we can minimize the functional $\operatorname{KL}(\cdot | \rho_\infty) + C$ with $C$ chosen in such a way that the constant in the $\delta \operatorname{KL}$ term cancels out.
    In conclusion, one can explicitly rewrite $\tilde\eta_m$ in \eqref{eq:fp_general_end} as a function of already known variables by
    \begin{equation}
        \tilde\eta_m = \left(\frac{\rho_\infty}{\hat\eta_m} \right)^{\frac{1}{1 + 2\beta_k}}. \label{eq:fp_terminal_KL}
    \end{equation}

    The fixed-point cycle defined above is schematically depicted in Figure~\ref{fig:fp_cycle}.
    We point out that formally replacing $\beta_k T_k$ with some positive $\epsilon$ and $\beta_k = 0$ yields the Sinkhorn algorithm~\cite{chen2016entropic}.
    \begin{figure}[ht]
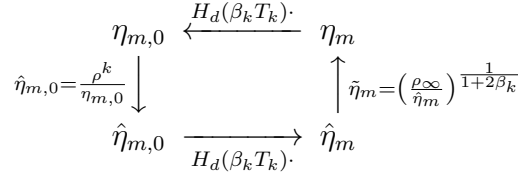

        \centering
        \[ \large
        \begin{CD}
            \eta_{m, 0}           @<{H_d(\beta_k T_k) \cdot }<<    \eta_{m} \\
            @V{\hat\eta_{m, 0}} = \frac{\rho^k}{\eta_{m,0}}VV             @AA{\tilde\eta_{m} = \left(\frac{\rho_\infty}{\hat\eta_{m}} \right)^{\frac{1}{1 + 2\beta_k}} 
}A \\
            \hat\eta_{m, 0}        @>>{H_d(\beta_k T_k) \cdot }>     \hat\eta_{m}
        \end{CD}
        \]
        \caption{Schematic of proposed fixed-point algorithm.}
        \label{fig:fp_cycle}
    \end{figure}

    For a system that converged (to a prescribed accuracy) after $M$ fixed-point iterations, the next density is given by
    \begin{equation}\label{eq:next_density}
        \rho^{k+1}(x) \approx \eta_M(x) \cdot \hat\eta_M(x).
    \end{equation}
    Additionally, one can determine the drift term of the Fokker-Planck equation describing the flow in~\eqref{eq:reg_dyn_JKO} from the optimality condition $v = \nabla_x \Phi$ by
    \begin{align}\label{eq:SDE_drift}
        v(x, t) &= 2 \beta_k \nabla_x \log{\eta(t, x)},\\ 
        \eta(t, x) &= (H^d(\beta_k (T - t)) \eta_{M})(x). \notag
    \end{align}
    Alternatively, using that
    \begin{multline*}
        \partial_t \rho + \nabla(\rho \nabla \Phi) - \beta_k \Delta \rho = \\ 
            =  \partial_t \rho + \nabla(\rho \nabla \Phi) - \beta_k \nabla \cdot (\nabla \rho)
            =  \partial_t \rho + \nabla(\rho \nabla \Phi) - \beta_k \nabla \cdot (\rho \nabla \log \rho) = \\
            =  \partial_t \rho + \nabla\left(\rho \nabla (\Phi - \beta_k \log \rho)\right) 
    \end{multline*} 
    and recalling the Hopf-Cole transform \eqref{eq:Hopf-Cole_inverse}, one can get an equivalent deterministic dynamics of the form
    \begin{equation*}\label{eq:continuity_equation_vtilde}
        \partial_t \rho + \nabla(\rho \tilde v) = 0 
    \end{equation*}
    with the drift term $\tilde v$ given by
    \begin{align}\label{eq:ODE_drift}
        \tilde v(t, x) &= \beta_k \nabla_x \left(\log \eta(t, x) - \log\hat\eta(t, x)\right), \\
        \hat\eta(t, x) &= \left(H_d(\beta_k t) \hat\eta_{0, M} \right)(x),\quad \eta(t, x) = \left( H_d(\beta_k (T_k - t))\eta_{M} \right)(x). \notag
    \end{align}

    \subsection{Convergence of the Minimizing Movements with regularized proximal}
    Convergence of the scheme~\eqref{eq:reg_dyn_JKO} is -- to the knowledge of the authors -- predominantly studied in the context of convergence of the discrete sequence $\{\rho^k \}_{k = 1}^\infty$ to the continuous solution of the gradient flow for the original energy functional $\mathcal{E}$ as both the time step and regularization parameter approach zero.
    In~\cite{carlier2017convergence} it is shown that in order to approximate the original gradient flow, the regularization has to decay sufficiently faster than the time step size.
    However, we are more interested in the rate of convergence of the discrete sequence to the minimizer of the energy functional, given a predefined sequence of steps and regularizations and possibly also making an informed choice of such a sequence.
    This problem can be attacked in the following way.
    In~\cite{cheng2023convergence} the EVI similar to~\eqref{eq:EVI} is derived for a sequence $\{ \rho^k \}$, where each term is not an exact solution of~\eqref{eq:JKO_general_definition} but a suboptimal one.
    To make this precise, we assume that the Wasserstein subgradient of the proximal functional at $\rho_{k+1}$ is bounded by some predefined sequence $\Xi_k$, i.e.,
    \begin{align*}
        \mathcal{F}_{k+1}(\rho) &:= \mathcal{E}(\rho) + \frac{1}{2\tau_k}W_2^2(\rho, \rho_k), \\
        \rho_{k+1}&: \exists\ \xi_{k+1} \in \partial_{W_2} \mathcal{F}_{k+1}(\rho_{k+1})\quad\text{s.t.}\quad \|\xi_{k+1}\|_{L_2(\rho_{k+1})} \leq \Xi_k.
    \end{align*}
    Then (again in case of a functional $\lambda$-convex along generalized geodesics) we deduce that
    \begin{equation}\label{eq:EVI_suboptimal}
        \left(1 + \frac{T_k\lambda}{2} \right)W^2_2(\rho_{k+1}, \rho_\infty) + 2T_k\left( \mathcal{E}(\rho_{k+1}) - \mathcal{E}(\rho_{\infty})\right) \leq W^2_2{(\rho_k, \rho_\infty)} + \frac{2T_k}{\lambda}\Xi_k^2.
    \end{equation}
    If applied by induction for instance for uniformly bounded $\Xi_k \leq \Xi$ and constant timestep $T_k = T$, it yields the linear convergence rate (\cite[Theorem 4.3]{cheng2023convergence})
    \begin{equation*}
        W^2_2(\rho_{k+1}, \rho_\infty) \leq \left(1 + \frac{T\lambda}{2}\right)^{-k} W^2_2(\rho_0, \rho_\infty) + \frac{4\Xi^2}{\lambda^2}.
    \end{equation*}
    We hypothesize that the subgradient norm $\Xi_k$ can be controlled by a proper choice of $(T_k, \beta_k)$ but leave this to future studies.
    
\section{Numerical solution}
    The main challenge for a numerical implementation of the method in the Eulerian approach is the number of degrees of freedom in the discretization, which is growing exponentially with the dimension $d$ of the problem. 
    We argue that this can be overcome by utilizing low-rank tensor methods, assuming that the problem exhibits some low-rank structure.
    This section provides the details on the low-rank methods used for the implementation.

    \subsection{Discretization and Tensor Train compression} 
    For a proof-of-concept implementation we consider a <<sufficiently large>> cube $\bigotimes_{k=1}^d[L_k; R_k]$ and a finite difference approximation on a regular grid with $N_k$ nodes in each dimension, namely
    \begin{equation*}
        h_k = \frac{R_k - L_k}{N_k},\qquad x_{k, i} = -L_k + h_k\cdot i.
    \end{equation*}
    A tensor with function values $g_\alpha$ with multiindex $\alpha$ is defined by
    \begin{equation*}
        g_{i_1\dots i_d} = g(x_{1,i_1}, \dots, x_{d,i_d}),
    \end{equation*}
    where $g$ is one of $\eta_m,\ \eta_{0, m},\ \hat\eta_{0, m},\ \hat\eta_{m}$.
    If the dimension $d$ of the problem is high, the amount of memory required to store all the values of the tensor increases exponentially, rendering it prohibitively complex in practice. 
    Hence, some compression approach has to be used and in the present work we focus on the Tensor Train format
    
    \begin{definition}{Tensor Train format}   
    A tensor $g_\alpha \in \mathbb{R}^{N_1 \times \dots \times N_d}$ is in the TT format if it has the form
        \begin{equation}\label{eq:tt_definition}
            g_{i_1\dots i_d} = G^1_{i_1 l_1}G^2_{l_1 i_2 l_2}\dots G^d_{l_{d-1}i_d},
        \end{equation}
        where for each $n \in \overline{1,\ d}$, $G^n_{l_{n-1}i_n l_n} \in \mathbb{R}^{r_{n - 1} \times N_n \times r_n}$ is a $3$-dimensional tensor with $r_0 = r_{d} = 1$ and Einstein summation convention is assumed.
        The values $r_k$ are called \emph{TT-ranks}.
    \end{definition}
    If $N = \max_n{N_n}$ and TT-rank $r = \max_n{r_n}$, accessing a component requires $\mathcal{O}(dr^2)$ operations and the storage complexity is $\mathcal{O}(dNr^2)$ compared to $\mathcal{O}(1)$ and $\mathcal{O}(N^d)$ in the case of storing all the values in a full tensor.
    Compression is achieved if the rank is significantly lower than the full rank.
    
    We assume that for each proximal step an adequate initial guess for the fixed-point $\eta_{0,i_1\dots i_d}$ in the TT format can be obtained.
    In the sequel it is explained how to approximate all the solution steps while staying in the TT format.
    
    \subsection{Solution of the heat equation in TT format}
    The solution of the heat equation amounts to applying the operator $H_d(s)$.
    It is required during the fixed-point iteration (see~\eqref{eq:heat_oper}-\eqref{eq:heat_oper_end}) and later to compute the drift terms for the sampling dynamics (\eqref{eq:SDE_drift}, \eqref{eq:ODE_drift}).
    We first replace the Laplace operator $\Delta$ with its second-order in space finite difference approximation
    \begin{equation}
        \label{eq:discrete_Lap}
        L_h = D_{1, h} \otimes I_2 \otimes \cdots \otimes I_d + 
            I_1 \otimes D_{2, h} \otimes \cdots \otimes I_d + \dots + 
            I_1 \otimes I_2 \otimes \cdots \otimes D_{d, h},
    \end{equation}
    where $\otimes$ denotes the tensor product, $I_k \in \mathbb{R}^{N_k \times N_k}$ is an identity matrix and 
    \begin{equation*}
        D_{n, h} = \frac{1}{h_n^2}
            \begin{pmatrix}
                -1 & 1      &           &        &         \\
                1  & -2     & 1         &        &         \\
                   & \ddots & \ddots    & \ddots &         \\
                   &        & \ddots    & \ddots & 1       \\
                   &        &           & 1      & -1        
             \end{pmatrix}
    \end{equation*}
    for $n \in \overline{1,\ d}$ is a $N_n \times N_n$ matrix, which represents the second-order finite difference approximation of $\frac{\partial^2}{\partial x_n^2}$.
    Then, for a time-dependent tensor $g(t)$ the following ODE holds true
    \begin{equation*}
        \frac{\mathrm{d}}{\mathrm{dt}}g(t) = \beta L_h g(t).
    \end{equation*}
    Its solution has the form
    \begin{equation}\label{eq:heat_solution}
        g(t) = e^{\beta_k t L_h}g(0).
    \end{equation}

    Due to the structure of the discrete Laplacian~\eqref{eq:discrete_Lap}, the matrix exponential of \eqref{eq:heat_solution} has the form (see~\cite[{Section~{2.4}}]{graham2018kronecker})
    \begin{equation}\label{eq:discrete_expm}
        e^{\beta_k t L_h} = e^{\beta_k t D_{1, h}} \otimes \cdots \otimes e^{\beta_k t D_{d, h}}.
    \end{equation}
    Finally, if the initial tensor $g(0)$ is in TT format, the solution $g(t)$ is given by
    \begin{multline*}
        g(t)_{i_1\dots i_d} 
            = \left(e^{\beta_k t L_h} \right)_{i_1\dots i_d j_1 \dots j_d}g(0)_{j_1 \dots j_d} = \\
            = \left(e^{\beta_k t L_h} \right)_{i_1\dots i_d j_1 \dots j_d} G^1_{j_1 k_1}G^2_{k_1 j_2 k_2}\dots G^d_{k_{d-1}j_d} 
            = \left(e^{\beta_k t D_{1, h}}_{i_1 j_1} \dots e^{\beta_k t D_{d, h}}_{i_d j_d}\right) G^1_{j_1 l_1}G^2_{l_1 j_2 l_2}\dots G^d_{l_{d-1}j_d} = \\
            = \left(e^{\beta_k t D_{1, h}}_{i_1 j_1} G^1_{j_1 l_1}\right) \left(e^{\beta_k t D_{2, h}}_{i_2 j_2} G^2_{l_1 j_2 l_2} \right)\dots \left(e^{\beta_k t D_{d, h}}_{i_d j_d} G^d_{l_{d-1}j_d}\right).
    \end{multline*}
    This means that the approximate solution can also be represented in TT format with new cores 
    \begin{equation}\label{eq:heat_tt_cores}
        e^{\beta_k t D_{n, h}}_{i_n j_n} G^n_{l_{n-1} j_n l_n} \in \mathbb{R}^{r_{n - 1} \times N_n \times r_n},   
    \end{equation}
    evidently having the same TT rank as the initial condition.
    The update of the TT cores by~\eqref{eq:heat_tt_cores} requires $\mathcal{O}(dN^2 r^2)$ operations compared to $\mathcal{O}(N^{2d})$ operations required do compute a general convolution of type $A_{i_1\dots i_d j_1\dots j_d} b_{j_1\dots j_d}$.
    For given $(\beta_k, T_k)$, $d$ matrix exponentials (or, in case $N_n,\ h_n$ in every direction are the same, only one) can be precomputed, for example by means of a Pad\'e approximation~\cite{almohy2010anew}.
    Thus we argue that the heat equation required during the fixed-point iteration can be efficiently solved in TT format without explicit time stepping.

    \subsection{Solution of initial and terminal conditions}
    Tensors $\hat\eta_{m, 0},\ \tilde\eta_m$ are defined by~\eqref{eq:fp_initial} and~\eqref{eq:fp_terminal_KL} in the sense that their value can be computed for any given index.
    We aim to acquire a TT representation of these tensors by computing the values at a much smaller subset of indices.
    Note that it is crucial to control the amount of evaluations of the right-hand side of~\eqref{eq:fp_terminal_KL}.
    This is because in this step the posterior evaluations take place, making it a computational bottleneck of the whole algorithm. 
    Such a task to construct a tensor from relatively few measurements is called \emph{tensor completion} or \emph{reconstruction}.
    A common algorithm is the well-known TT-cross approximation~\cite{oseledets2010tt}, for which we provide a brief sketch.
    Given a multi-dimensional tensor $T\in \mathbb{R}^{n_1 \times \dots \times n_d}$, it may be of use to view it as a $2$-dimensional matrix, produced by grouping indices of several dimensions $(i_1, i_2, \dots i_k)$ into a joint index $i_1i_2\dots i_k \in \overline{1, n_1 \cdots n_k}$, by some one-to-one mapping.
    This matrix is then given by
    \[
        \left[T^{(k)}\right]_{i_1i_2\dots i_k, i_{k+1}\dots i_{d}} = T_{i_1,\dots i_d}
    \]
    and called the \emph{unfolding} of the tensor.
    If a tensor has a best TT approximation error $\epsilon$ with ranks $r_k$, the method reconstructs it by recursively building skeleton decompositions of the aforementioned unfoldings.
    For a given sequence of indices for the skeleton decomposition $\{(I^{\leq k}, I^{>k}) \}_{k = \overline{1, d - 1}}$, where each $I^{\leq k} = \{(i^j_1, i^j_2, \dots i^j_k) \}$, the error can be estimated as follows.
    \begin{theorem}[{\cite[Theorem 2]{qin2022error}}]
        For a given  tensor $\mathcal{T} \in \mathbb{R}^{N_1 \times \dots \times N_d}$, assume there is  a tensor $\mathcal{T}_{r_k}$ in TT format with ranks $r_k$ such that
        \begin{gather*}
            r = \max_{k = \overline{1, d-1}} r_k, \quad
            \epsilon = \max_{k = \overline{1, d-1}} \|T^{\langle k \rangle} - T_{r_k}^{\langle k \rangle}\|_F, \\
            \kappa = \max_{k = \overline{1, d-1}} \left\{ %
                        \|T^{\langle k \rangle}(:, I^{>k})\cdot  (T^{\langle k \rangle}(I^{\leq k}, I^{>k}))^{-1} \|,%
                        \|(T^{\langle k \rangle}(I^{\leq k}, I^{>k}))^{-1} \cdot T^{\langle k \rangle}(I^{\leq k}, :) \| \right\}.
        \end{gather*}
        Then, for sufficiently small $\varepsilon>0$ there exists a TT cross approximation $\mathcal{T}_{\text{TT-cross}}$ such that
        \begin{equation}
            \|\mathcal{T} - \mathcal{T}_{\text{TT-cross}}\|_F \leq \frac{(3\kappa)^{\lceil \log_2 N \rceil} - 1}{3\kappa - 1}(r + 1)\varepsilon.
        \end{equation}
    \end{theorem}
    One pass of computing the TT cores for each dimension by building skeleton decompositions is referred to as a \emph{sweep}.
    It requires $\mathcal{O}(dNr^2)$ evaluations of the right-hand side tensor elements.
    Practically, the optimal indices are not known, thus, several sweeps are performed, using the approximation from the previous ones and various heuristics to improve the selection of the indices. 
    The description of these approachse can be found in~\cite{dolgov2020parallel} and references therein.
    For an arbitrary tensor, the low-rank approximation is never guaranteed, but the maximum rank can be estimated according to the knowledge of the problem and numerical resources available.
    The rank of the cross-approximation can be dynamically adapted with the DMRG technique~\cite{savostyanov2011fast}.

    A hallmark of the Eulerian approach is that by using a fixed grid, the evaluations of the posterior are carried out in a comprehensive and ``organized'' way.
    More specifically, since in our problem the posterior can be evaluated only at a finite (although extremely large even for moderated dimensions $d > 4$) number of points determined in advance by the grid, we can implement a caching strategy in practice to reduce the number of evaluations drastically.
    In the current implementation, we simply store the first $N_{\text{cache}}$ calls to the posterior.
    The cache is shared between inner fixed-point iterations and outer proximal steps.
    In the numerical experiments below we demonstrate that this feature indeed significantly decreases the amount of the posterior calls required to solve the problem.
    
    \subsection{Fixed-point iteration}
    As explained in \Cref{ss:fp_equations}, the PDE coupling problem of one proximal step is recast as a fixed-point problem
    \[
        \eta^* = G(\eta^*) 
    \]
    with TT representation $\eta$ of the potential $\eta(T_k, x)$ and operator $G$ is the approximation of the cycle defined in \eqref{eq:fp_general}-\eqref{eq:fp_general_end} and depicted in Figure~\ref{fig:fp_cycle}.
    The numerical approximations are described in the preceding sections.
    For an iterative solution procedure, consider a sequence
    \[
        x_m = \eta_m,\quad g_m =G(x_m),\quad r_m = g_m - x_m.
    \]
    The simplest approach is the relaxation method 
    \[
        x_{m+1} = q_m g_m + (1 - q_m) x_m,
    \]
    where $q_m \in (0;, 1]$ is called \emph{relaxation factor}.
    It can either be constant or selected adaptively.
    If constant, setting smaller values of $q_m$ improves the stability of the method but deteriorates its speed.
    For $q_m \equiv 1$, one recovers the Picard method.
    Aitken's scheme~\cite{birken2015fast} updates $q_m$ adaptively depending on two previous residuals.
    However, it is evident that in the region where the operator $G$ is a contraction, this scheme cannot converge faster than the Picard method with $q = 1$.
    In our preliminary experiments, the fixed-point problems appear to always converge for the Picard method with $q = 1$ and hence the application of Aitken's scheme is not reported.

    A more involved acceleration method called \emph{Anderson Acceleration} (AA) relies on the minimization in the affine hull of $P$ previous residuals.
    It is described by
    \begin{align}\label{eq:anderson_fp}
        (\alpha_0,\ \dots,\ \alpha_{P-1}) &= \arg \min_{\alpha_0 + \dots + \alpha_{P-1} = 1} \left\|\sum\limits_{i=0}^{P-1} \alpha_i r_{m-i}\right\|_2^2, \\
        x_{m+1} &= q \sum\limits_{i=0}^{P-1} \alpha_i g_{m-i} + (1-q) \sum\limits_{i=0}^{P-1} \alpha_i x_{m-i}. \notag
    \end{align}
    The theoretical results on the convergence of accelerated FP in general can be found in~\cite{park2022exact} and for AA in particular in~\cite{toth2015convergence}. 
    AA can be related to multisecant methods~\cite{fang2009two}.
    In fact, for a linear problem it is equivalent to GMRES~\cite{walker2011anderson}. 
    Practically, AA has been reported to improve both speed and robustness of fixed-point iteration compared to the Picard iteration (see~\cite{aksenov2021application,tang2022accelerating} and references therein).
    For the application with TT representations of the iterates, each fixed-point update is followed by a TT truncation up to prescribed rank $r_{\text{max}}$ and tolerance $\varepsilon$.

%
%

    \subsection{Sampling the approximate solution}\label{ss:tt_sampling}
    The PDE-based approach of our method combines the Lagrangian and Eulerean viewpoint since the approximate solution of the PDE system defines a dynamics that describes the evolution of particles, moving from a start towards an updated distribution.
    The representation \eqref{eq:ode_characteristic_method} provides a deterministic description of the dynamics, resulting in a deterministic sampling method. 
    It requires solving an ODE with right-hand side defined by~\eqref{eq:ODE_drift}, i.e., 
    \begin{gather}\label{eq:ode_interpolation_dynamic}
        \dot{X}(t) = \beta_k \nabla_x\big(\log \eta(t, X(t)) - \log\hat{\eta}(t, X(t))\big),\qquad
        X(0) \sim \rho_0. \notag
    \end{gather}
    Similar methods have recently become popular, for instance in the generative modelling community~\cite{song2020score}.
    The approach only requires random samples of the initial density, which can be made tractable by construction.
    Note that the ODE can be solved in parallel with high precision and with an adaptive choice of time steps for each trajectory. 
    Alternatively, from the Fokker-Planck equation~\eqref{eq:RegWass_proximal_rho} and the fact that $\Phi = 2\beta_k \log\eta$, an SDE can be defined by 
    \begin{gather}\label{eq:sde_interpolation_dynamic}
        \mathrm{d}X_t = 2\beta_k\nabla{\log\eta(X_t, t)}\mathrm{d}t + \sqrt{2\beta_k} \mathrm{d}W_t,\qquad
        X_0 \sim \rho_0, \notag
    \end{gather}
    with standard Brownian motion $W_t$ and the law of $X_t$ denoted by $\rho_t$. 
    Any suitable numerical method of the SDE solution as for example Euler-Maruyama provides an algorithm for approximate sampling based on this dynamics.

    \begin{remark}
        Our preliminary computations show that it might be beneficial to employ a combination of both dynamics.
        The ODE dynamic can be approximately integrated with high-order adaptive methods such as $\operatorname{RK45}$.
        In our sampling tasks, the steps taken are quite large in the beginning of the time interval.
        However, towards the terminal time for a small subset of the trajectories the step size estimated by the adaptive solver becomes extremely small.
        Additionally, due to the discretization error in the regions of low probability and approximate computation of the gradient, the latter can become too small and particles might get stuck in such regions.
        This leads to samples produced solely with the ODE scheme to exhibit certain unwanted artifacts.
        Heuristically, some scheme with addition of noise would push the particles away from the regions with zero gradient.
        Thus, we suggest to combine the deterministic and the stochastic dynamics.
        More precisely, for some predefined value $\epsilon_{\text{SDE}}$ we solve the ODE system \eqref{eq:ode_interpolation_dynamic} on the time interval $[0; (1 - \epsilon_{\text{SDE}})T_k]$ (with outer step size $T_k$ in the proximal scheme).
        Then, the final $\epsilon_{\text{SDE}}T_k$ time is integrated with the Euler-Maruyama method with a predefined number of steps.
        The SDE time fraction $\epsilon_{\text{SDE}}$ is chosen heuristically and takes values $0.01\sim0.001$
    
        \begin{figure}
            \centering
            \includegraphics[width=.8\textwidth]{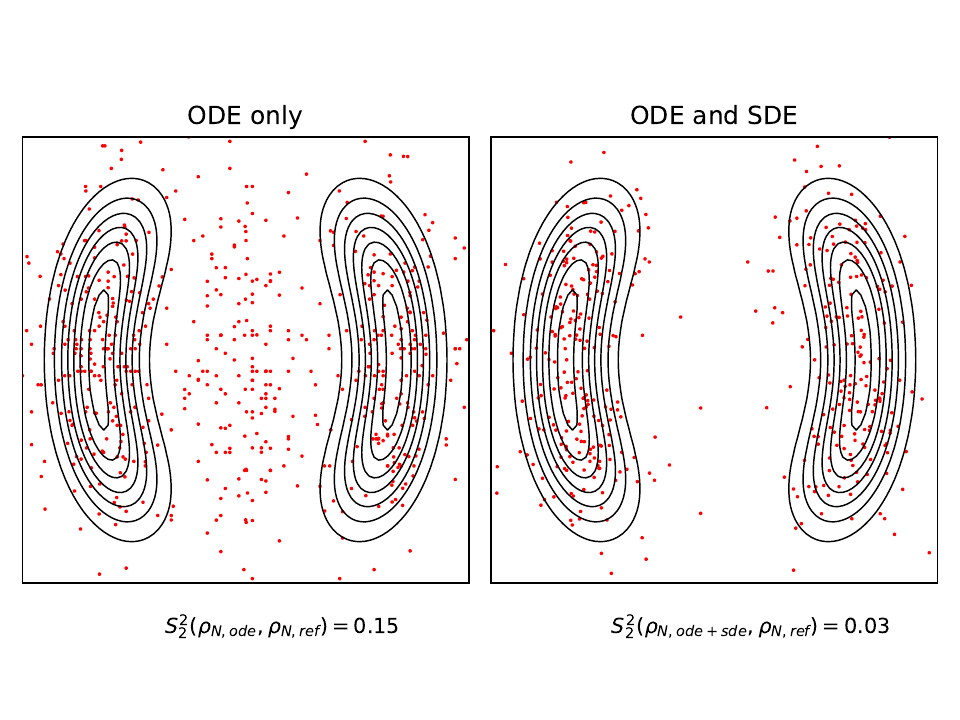}
            \caption{Contour lines of the test distribution and the samples generated with the deterministic method (left) and the combination of deterministic and stochastic method (right).}
            \label{fig:sampling_artefact}
        \end{figure}
    
        Figure~\ref{fig:sampling_artefact} illustrates the effect discussed above.
        The plot on the left depicts the sample produced with the ODE dynamics only.
        It clearly exhibits a ``line-like'' artefact in the center, where the actual probability density of the target is rather low.
        The sample generated with the combination of ODE and SDE dynamics (right) fits the reference distribution better (for this particular example, the Sinkhorn distance is roughly $5$ times lower).
        In fact, it also requires less computational time because the region where the ODE solver requires small time steps is avoided.
    
        Post-processing procedures for deterministic sampling methods that involve the addition of noise are quite common.
        For example,~\cite{song2020score} suggests using a predictor-corrector scheme, where the solution of the probability flow ODE is corrected at each step by a step of Langevin dynamics towards the current distribution.
        A Langevin post-processing with multiple steps is also reported in~\cite{sommer2024generative}.
        The joint dynamics presented in the current work provides a novel viewpoint on this established computational know-how. 
    \end{remark}

    \section{Numerical experiments}
    \label{sec:experiments}
    This section is concerned with the numerical assessment of the method proposed above. 
    First, the properties of the method are explored with a trivial example of fitting a Gaussian target with diagonal covariance matrix.
    The dependence of the convergence properties of the fixed-point iteration on the stepsize and regularization parameter is studied. 
    We demonstrate that in principle a single proximal step can be sufficient to provide a good approximation. 
    This comes at the cost of increasing complexity of the fixed-point iteration, which in turn can be mitigated by implementing accelerated fixed-point schemes.

    Second, the method is tested against the baseline Metropolis-Hastings MCMC method on several synthetic problems. 
    We highlight the performance of the method in case of multimodal distributions and distributions with nonconvex potentials, which pose well-known challenges for the baseline method.
    We carefully estimate the <<computational budget>> of our method, expressed as the number of the calls to the posterior, and demonstrate that our posterior approximation is superior to the baseline given the same amount of posterior calls.
    
    For the Python implementation, the \texttt{teneva} package~\cite{chertkov2023black} is employed, which contains all basic tensor-train operations as well as a cross-approximation solver. 
    
    \subsection{Approach to performance assessment}
    Before providing the numerical applications of the described method, we briefly discuss the approaches used to comprehensively assess its practical performance.
    The method itself allows to track the $\operatorname{KL}$ divergence to the posterior by carrying out an additional cross-approximation of the quantity
    \begin{equation}\label{eq:tt_kl_wpost}
        \log\left( \frac{\rho_{\text{TT}}}{\rho_\infty} \right) \rho_{\text{TT}}
    \end{equation}
    on a grid and a successive approximation of the integral by contracting the low-rank TT approximation.
    Note that $\rho_{\text{TT}} = \eta_{M} \cdot \hat\eta_{M}$ and for the converged solution, from \eqref{eq:fp_terminal_KL}
    \begin{equation*}
        \log\left( \frac{\rho_{\text{TT}}}{\rho_\infty} \right) = -2\beta_k \log \eta_{M}.
    \end{equation*}
    We hence suggest that the divergence can actually be evaluated without additional posterior calls. 
    The $\operatorname{KL}$ divergence estimated in this fashion can be utilized for the assessment of the convergence of the method.
    
    Considering the targeted practical application in Bayesian inversion and related applications, estimating the expectations of various quantities of interest is of importance. 
    Certain functions such as $x_i$ and $x_i x_j$ do not depend on the solution of the forward model and have a straightforward low-rank representation. 
    Thus, their moments (e.g. means and covariances of the parameters) can be estimated with negligible computational costs with {TT} contractions.  
    For more general quantities, an approximation similar to the one discussed above in the context of estimating the $\operatorname{KL}$ divergence is possible.
    However, this possible approach is left for the future studies. 
    Instead, we focus on using the acquired model to sample from the approximate posterior in order to relate and compare the method to <<Eulerian>> methods that appear to be more established in the area.
    
    We aim to compare the produced samples either to the samples from the posterior (for tractable ones such as Gaussian mixtures) or to a baseline sampling method (such as Langevin dynamics or MCMC).
    Several issues are to be dealt with in order to do so.
    First, the empirical measure associated with a sample,
    \[
        \rho_N = \sum\limits_{i=1}^N \delta_{x_i},\ x_i \sim \rho
    \]
    is obviously different from the original measure $\rho$.
    Thus, there is some nonzero Wasserstein distance between them.
    It can be observed that two different samples from the same distribution can have quite a large Wasserstein distance, especially for higher dimensions of the problem.
    This suggests that when the approximate sample is sufficiently close to the target, the error is dominated by the finite-sample effect and not with the approximation error. 
    This effect can be somewhat mitigated by using a larger sample size.
    Second, the computational complexity of the original linear OT problem~\eqref{eq:wasserstein_distance} does not scale well with the number of samples, requiring alternatives to be considered.
    One way to tackle these issues is to solve the \emph{entropic OT} distance given by
    \[
        W^2_{2, \varepsilon}(\rho_1,\ \rho_2) = \min\limits_{\pi \in \Pi(\rho_1, \rho_2)} \int \|x - y\|^2 
                                                    + \varepsilon \operatorname{KL}\left(\pi | \rho_1 \otimes \rho_2 \right),
    \]
    which can be done efficiently.
    The entropic regularizer adds a <<bias>>, i.e., $W^2_{2,\varepsilon}(\rho, \rho) \neq 0$ for general $\rho$.
    Introducing <<debiasing>> terms, one gets a so-called \emph{Sinkhorn divergence} by
    \[
        S^2_{2,\varepsilon}(\rho_1, \rho_2) = W^2_{2, \varepsilon}(\rho_1,\ \rho_2) 
            - \frac{1}{2}\left( W^2_{2, \varepsilon}(\rho_1,\ \rho_1) + W^2_{2, \varepsilon}(\rho_2,\ \rho_2) \right).
    \]
    This divergence is an approximation of $W^2_2$ up to $O(\varepsilon^2)$~\cite{chizat2020faster} and can be efficiently computed with the Sinkhorn algorithm~\cite{cuturi2013sinkhorn} or by means of stochastic optimization~\cite{genevay2016stochastic} in higher dimensions. 
    In our numerical experiments, the implementation of the Sinkhorn algorithm from the \texttt{geomloss} package~\cite{feydy2019interpolating} is used.

    An alternative approach to measure the OT distance betwen samples would be to use the \emph{sliced Wasserstein} distance given by
    \begin{gather*}
        \overline{W}^2_2(\rho_1, \rho_2) = \sup\limits_{\theta \in \mathbb{S}^{d-1}} W^2_2(\theta_\sharp \rho_1, \theta_\sharp \rho_2), \\
                                                \theta_\sharp = (\langle \theta, \cdot \rangle)_\sharp,
    \end{gather*}
    where $\theta \in \mathbb{S}^{d-1}$ is a direction and $\theta_\sharp$ denotes a pushforward of the measure with the map that projects vectors to the direction $\theta$.
    It is shown in~\cite{bayraktar2021strong} that while $\overline{W}^2_2$ induces a metric equivalent to $W^2_2$, the computation of it relies on solving unidimensional OT problems for which a closed-form solution exists. 
    The implementation of the sliced Wasserstein distance from \texttt{POT} package~\cite{flamary2021pot} was employed, where $\sup$ is estimated by taking a finite number of random projections. 
    Except for the numerical efficiency, the sliced OT metric does not show any qualitative benefits compared to the Sinkhorn distance in our test.
    Hence, in the sequel the results are only reported in the latter distance.

    We aim to show that our method produces samples that are consistent with the target measure by comparing them to the reference samples from the latter.
    Due to the finite-sample effect discussed in the beginning of this section, optimal transport distances between the samples even from the same distribution are in fact random variables not identical to zero.
    Following~\cite{eigel2022less}, we suggest comparing the distributions of this random variables.
    More specifically, we consider the distribution of $S^2_{2,\varepsilon}(\rho^1_{N_1},\ \rho^2_{N_2})$, where $\rho^i_N = \frac{1}{N}\sum_{k=1}^N \delta_{x_k}$ and $x_k \sim \rho^1 \text{i.i.d.}$, i.e. the empirical measure defined by an {i.i.d.} sample of size $N$ from the measure $\rho^1$. 
    If the measures $\rho^1$ and $\rho^2$ are actually equal, the distributions of $S^2_{2,\varepsilon}(\rho^1_{N_1},\ \rho^2_{N_2})$ and $S^2_{2,\varepsilon}(\rho^1_{N_1},\ \rho^1_{N_2})$ should coincide.
    Based on this, we acknowledge that our method generates samples from the target measure if the distribution of the Sinkhorn distances to the reference samples is close as a distribution to the distribution of Sinkhorn distances between independent reference samples.

    The rest of the section is organized as follows.
    The first part deals with verification computations for one step of the method with a Gaussian distribution with diagonal covariance matrix as posterior.
    The goal is to analyze the behavior of the method for different combinations of values of the parameters $T,\ \beta$ and the acceleration of the fixed-point scheme.
    In the second part, the method is compared to a baseline Metropolis-Hastings MCMC method for the task of sampling.
    Test distributions include multimodal distributions and a distribution with nonconvex potential.

    \subsection{Verification}
    We perform verification computations for a simple setting where both $\rho_0$ and $\rho_\infty$ are Gaussian distributions.
    For this, assume a uniform grid on the hypercube $[-L, L]^d,\ L = 3$ with $N = 30$  nodes in each dimension.
    Moreover, $\rho_0 = \mathcal{N}(0, \operatorname{I}_d)$ and $\rho_\infty = \mathcal{N}(m, \sigma\operatorname{I}_d)$ with $m$ generated randomly with $m_i \sim U\left([-L/2,\ L/2]\right)$ and $\sigma = 0.5$. 
    The dimension of the problem is $d = 16$.
    We track the $\operatorname{KL}$ divergence as discussed above. 

    A single step of the method is performed for a broad range of parameters $T$ and $\beta$.
    The results are presented in \tabref{tab:one_step_normal}
    \begin{table}[ht]
        \centering
        \begin{tabular}{lrrrrr|rrrrr}
\toprule
 & \multicolumn{5}{c}{$\operatorname{KL}(\rho_{\text{TT}}|\rho_\infty)$} & \multicolumn{5}{c}{$N_{fp}$} \\
$\beta$ & $1$ & $0.1$ & $0.01$ & $0.001$ & $0.0001$ & $1$ & $0.1$ & $0.01$ & $0.001$ & $0.0001$ \\
$T$ &  &  &  &  &  &  &  &  &  &  \\
\midrule
$0.1$ & $9.1$ & - & - & - & - & $13$ & - & - & - & - \\
$1$ & $5.7$ & - & - & - & - & $13$ & - & - & - & - \\
$10$ & $6.7$ & $0.13$ & - & - & - & $13$ & $58$ & - & - & - \\
$100$ & $6.8$ & $0.14$ & $0.0014$ & - & - & $13$ & $42$ & $4 \cdot 10^{2}$ & - & - \\
$1000$ & $6.7$ & $0.14$ & $0.0016$ & $1.7 \cdot 10^{-5}$ & - & $13$ & $42$ & $3.7 \cdot 10^{2}$ & $1 \cdot 10^{3}$ & - \\
$10000$ & $6.6$ & $0.14$ & $0.0016$ & $1.8 \cdot 10^{-5}$ & $2.6 \cdot 10^{-6}$ & $13$ & $42$ & $3.7 \cdot 10^{2}$ & $1 \cdot 10^{3}$ & $1 \cdot 10^{3}$ \\
$100000$ & $6.7$ & $0.14$ & $0.0016$ & $1.8 \cdot 10^{-5}$ & $2.6 \cdot 10^{-6}$ & $13$ & $42$ & $3.7 \cdot 10^{2}$ & $1 \cdot 10^{3}$ & $1 \cdot 10^{3}$ \\
\bottomrule
\end{tabular}

        \smallskip
        \caption{Approximated $\operatorname{KL}$ and fixed-point complexity with Picard iteration for one step of the method.}
        \label{tab:one_step_normal}
    \end{table}
    One can see that with small values of $\beta T$, which plays the role of the effective time horizon in the heat equations, the method does not converge.
    When this value is sufficiently large, a rather good reconstruction can be achieved in one step for small $\beta$. 
    The minimal value of $\operatorname{KL}(\rho_\text{TT} | \rho_\infty) \sim 10^{-6}$ is obtained for $\beta = 10^{-4},\ T=10^5$.
    The fixed point iterations are performed until the relative change of the discretized variables $\frac{\|\eta_m - \tilde\eta_m \|_F}{\|\eta_m \|_F}$ is less than $10^{-5}$.
    The convergence of this quantity is linear with the rate decreasing when $\beta$ decreases.
    The amount of fixed-point iterations can quickly become large, motivating the introduction of acceleration methods.

    In our case, for acceleration the Anderson method based on $P = 2$ iterates is chosen because it admits a closed form solution to the minimization subproblem. 
    \begin{figure}[ht]
        \begin{subfigure}[t]{0.45\textwidth}
            \centering
            \includegraphics[width=\textwidth, trim=0 20 0 0, clip]{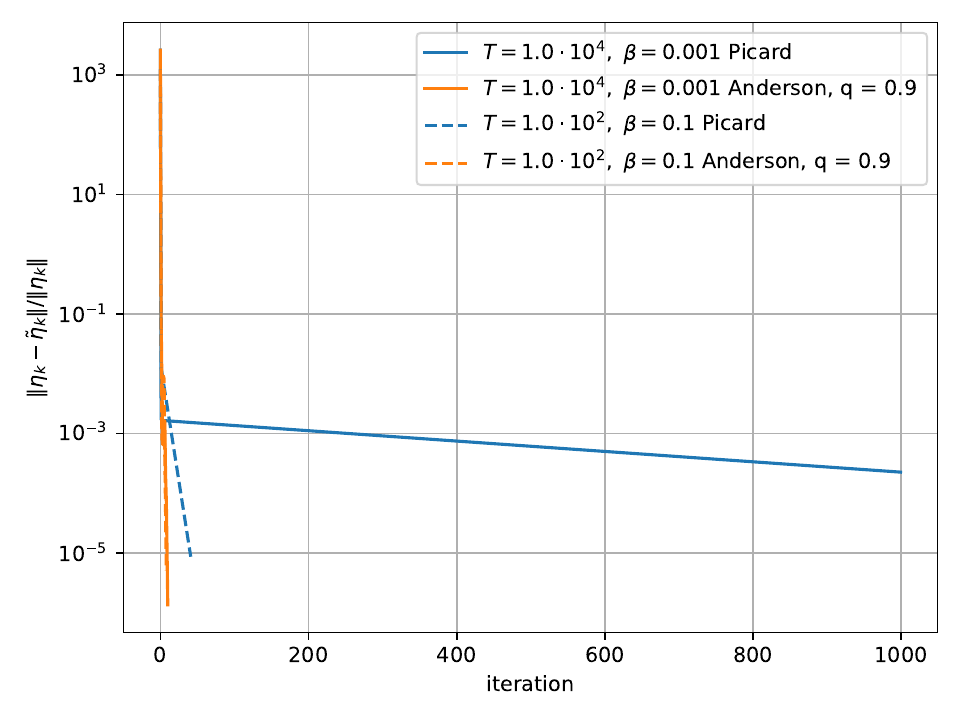}
            \caption{Relative error during fixed-point iterations for one step of the method.}
            \label{fig:fp_convergence_verification}
        \end{subfigure}
        \hfill
        \begin{subfigure}[t]{0.53\textwidth}
            \centering
            \includegraphics[width=\textwidth, trim=0 -40 0 0]{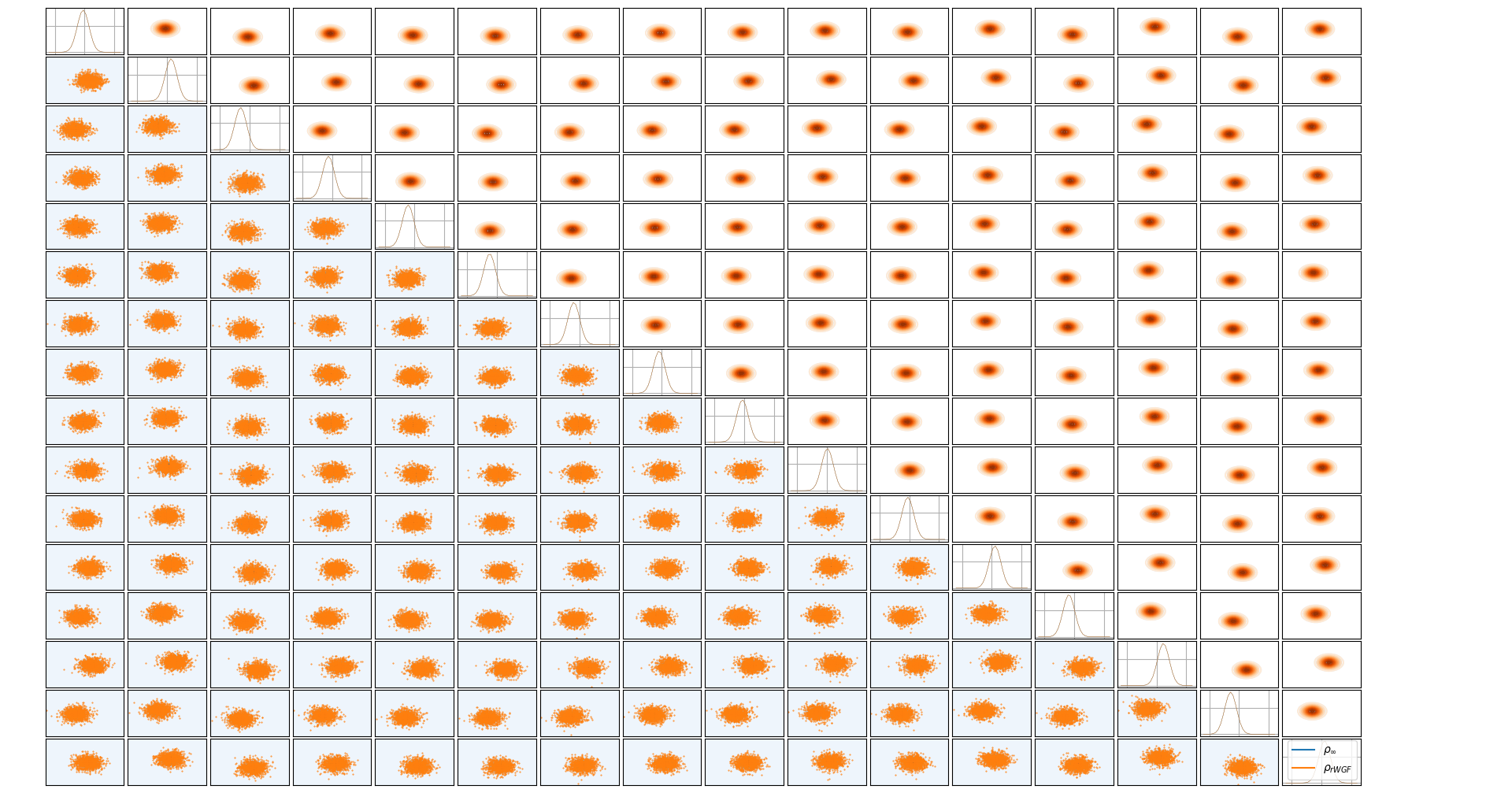}
            \caption{TT approximation after one step with $T~=~10^{5},\ \beta~=~10^{-4}$.}
            \label{fig:normal_marginals_verification}
        \end{subfigure}
        \caption{Reconstruction for a Gaussian target distribution.}
    \end{figure}

    The example of the fixed-point convergence plot is presented in Figure \ref{fig:fp_convergence_verification}.
    For illustrative purposes, we present the convergence results for two problems with the same value of $\beta T = 10$ and higher and lower regularization, $\beta = 10^{-1}$ (dashed lines) and $\beta = 10^{-3}$ (solid lines) correspondingly in the same plot.
    For the Picard iteration, one can clearly see that the linear convergence rate (slope of the line on the plot) is larger for larger regularization.
    Accelerated fixed-point iteration (orange) provides superlinear convergence and requires at maximum tens of iterations, giving a speed-up of up to $3$ orders of magnitude while potentially reaching higher fixed-point tolerance. 
    
    A downside of the Anderson acceleration is the presence of the relaxation parameter $q$.
    From our numerical experiments, it cannot always be set to $1$ in this problem since the computation may diverge.
    It is known that small $q$ reduce convergence speed and to the authors' knowledge, there is no way to select it a priori or adjust it adaptively.
    Despite of this, in the further experiments Anderson acceleration with $P = 2$ and $q$ around $0.8-0.9$ is used.

    \Cref{fig:normal_marginals_verification} depicts the approximate distributions. 
    Plots on the diagonal represent the marginal distributions of each variable $x_i$.
    The upper right triangle shows the contour lines of joint marginal distributions of variables $(x_i, x_j)$ of the approximate and the true posterior.
    In the bottom left triangle, the contour lines belong to the true posterior and the scatter plot represents the sample from the approximate posterior acquired with the deterministic ODE method as described in \Cref{ss:tt_sampling}.
    
    \subsection{Sampling from test distributions}
    The main focus of this part is to study the sampling capabilities of our method and compare the samples to baseline sampling methods using metrics based on optimal transport.
    The choice of the test distributions is partially inspired by~\cite{han2024tensor}.
    In addition to the Gaussian mixture family, we consider distributions described by their potential $V:\rho \propto e^{-V}$.
    Multimodal distributions and distributions with nonconvex $V$ can be particularly challenging for classical sampling methods such as MCMC or Langevin dynamics.
    
    The chosen test posteriors are given by
    \paragraph{Gaussian mixture}
    \[
        \rho_{\text{GM}} \propto \sum\limits_{k=1}^K w_k\rho_{m_k, C_k},
    \]
    where $\rho_{m_k, C_k}$ is a normalized density of a $d$-dimensional Gaussian distribution with mean $m_k$ and covariance $C_k$, $w_k$ are positive weights and $K$ is the number of components.
    In our test, we set $d = 30,\ K = 5, C_k = \sigma^2 I_d,\ \sigma^2 = 0.5$, weights are identical and $m_{i,j}$ is generated randomly from $U([-L/2, L/2])$.

    The next two distributions are also inspired by~\cite{han2024tensor}.
    Due to multimodality (in the former case) and the non-convexity of the potential (in the latter), they are difficult to sample with MCMC methods.
    \paragraph{Double-moon potential distribution}
    \[
        \rho_{\text{DM}}(x) \propto \exp\left( -2(\|x\|_2 - a)^2\right) \left( \exp(-2(x_1 -a)^2) + \exp(-2(x_1 +a)^2) \right)
    \]
    with scalar parameter $a = 2$
    \paragraph{Nonconvex potential distribution}
    \[
        V_{\text{NC}}(x) = \left( \sum\limits_{i = 1}^d \sqrt{|x_i - a_i|} \right)^2 
    \]
    with means vector $a = (a_i, \dots, a_d)$ taken as $a_i = (-1)^i$.

    The experiments are organized as follows. 
    The Tensor-Train approximation is computed with one regularized JKO step.
    During the fixed-point iterations, the calls to the posterior are cached and the number of actual calls to the posterior and those taken from the cache is kept track of.
    The sample of size $n_s = 400$ is then generated with the composition of ODE and SDE dynamics as described in~\Cref{ss:tt_sampling}.
        
    For comparison, reference samples are generated.
    In case of a Gaussian mixture, the target is sampled exactly.
    For the two other distributions, the reference is generated with Metropolis-Hastings MCMC (\texttt{emcee} package\cite{foreman2013emcee} is chosen as a particular implementation).
    The autocorrelation time $\tau_{ac}$ of the parameters can be estimated within the package, following the approach in~\cite{goodman2010ensemble}.
    We run $n_s = 400$ independent chains with starting value distributed as $\mathcal{N}(0, I_d)$ for a sufficiently large number of iterations ($\gtrsim 10000$) so that $N_{\text{iter}} > 50 \tau_{ac}$.
    The covariance parameter in the Gaussian step is chosen so that the acceptance rate is close to $25\%$.

    Finally, to have some perspective on the computational efficiency of the method, we compare it to the baseline method.
    In order to do this, MCMC is run again with the same parameters but with a shorter chain and the number of iteration chosen so that the number of the posterior calls is equal to the number of \emph{unique} calls carried out with the TT method.
                
    The procedure is repeated $20$ times for the TT method and the short MCMC chain.
    We also generate $20$ reference samples (in case of the long MCMC chain, $20$ states separated by $t_{ac}$ iterations are taken from the end of the chain). 
    We then compute $S^2_2(\rho^{\text{TT}}_{n_s}, \rho^{\text{ref}}_{n_s}),\ S^2_2(\rho^{\text{MCMC}}_{n_s}, \rho^{\text{ref}}_{n_s}),\ S^2_2(\rho^{\text{ref}}_{n_s}, \rho^{\text{ref}}_{n_s}),\ $ for each pair of different samples.
        
    \begin{table}[ht]
        \centering
        \begin{tabular}{lccccccccc}
\toprule
Distribution & $d$ & $r_{max}$ & \multicolumn{2}{c}{$N_\infty$} & \multicolumn{3}{c}{$S^2_{\varepsilon}$} & \multicolumn{2}{c}{Double OT} \\
 &  &  & Unique & Total & Ref. to ref. & Ref. to TT & Ref. to MCMC & TT & MCMC \\
\midrule
Mixture & $30$ & 5 & $1290K$ & $2561K$ & $9.04 \pm 0.10$ & $9.05 \pm 0.14$ & $9.77 \pm 0.33$ & $2.1 \cdot 10^{-3}$ & $5.8 \cdot 10^{-1}$ \\
Double-Moon & $6$ & 3 & $111K$ & $164K$ & $0.65 \pm 0.02$ & $0.67 \pm 0.03$ & $0.66 \pm 0.03$ & $3.4 \cdot 10^{-4}$ & $1.2 \cdot 10^{-4}$ \\
Nonconvex & $6$ & 2 & $25K$ & $183K$ & $0.06 \pm 0.01$ & $0.08 \pm 0.01$ & $1.75 \pm 0.11$ & $3.9 \cdot 10^{-4}$ & $2.9 \cdot 10^{0}$ \\
\bottomrule
\end{tabular}

        \smallskip
        \caption{Summary of results for each distribution, showing problem dimension, maximal TT rank, posterior call counts, average Sinkhorn distance with standard deviation for each sampling method, and the double OT estimate between distributions.}
        \label{tab:sec5_result}
    \end{table}
    The results are presented in  \tabref{tab:sec5_result}. 
    For each distribution, the dimension of the problem and the maximal TT rank of the solution $r_{max}$ are presented.
    $N_\infty$ denotes the number of the posterior calls carried out by the TT method, including the unique calls (when the posterior is actually called) and the total number (i.e. unique and re-used from cache).
    The Sinkhorn distance $S^2_\varepsilon$ averaged with respect to every different pair of samples generated by the TT method, by the short MCMC chain and between the reference itself, is presented alongside its standard deviation.
    Finally, the OT distance between these distributions is denoted by Double OT.
    The results are illustrated in Figure~\ref{fig:sec5_plot}.
    The contour lines of a two-dimensional marginal of the test distribution are shown alongside with the reference sample, a sample generated with the TT method and the  MCMC sample with a shorter chain generated with the same amount of posterior calls as for the TT method.

    Interestingly, the test runs show that approximation of multimodal and concentrated distributions is in fact possible with quite low TT ranks.
    One can see that the cache plays a significant role in the computation, increasing the effective amount of the posterior calls (compared to the scheme with no cache) by a factor ranging from $\sim 1.5$ times in the case of the Double-Moon problem up to $\sim 7.3$ times in the case of the Nonconvex problem.
    In general, the method performs at least as good as the baseline with the same computational effort.
    We note that the Nonconvex problem was particularly difficulty for the MCMC method, where the difference between the short chain MCMC sample and the reference can be seen immediately in Figure~\ref{fig:sec5_plot}. 
    \begin{figure}[ht]
        \centering
        \includegraphics[width=.85\textwidth]{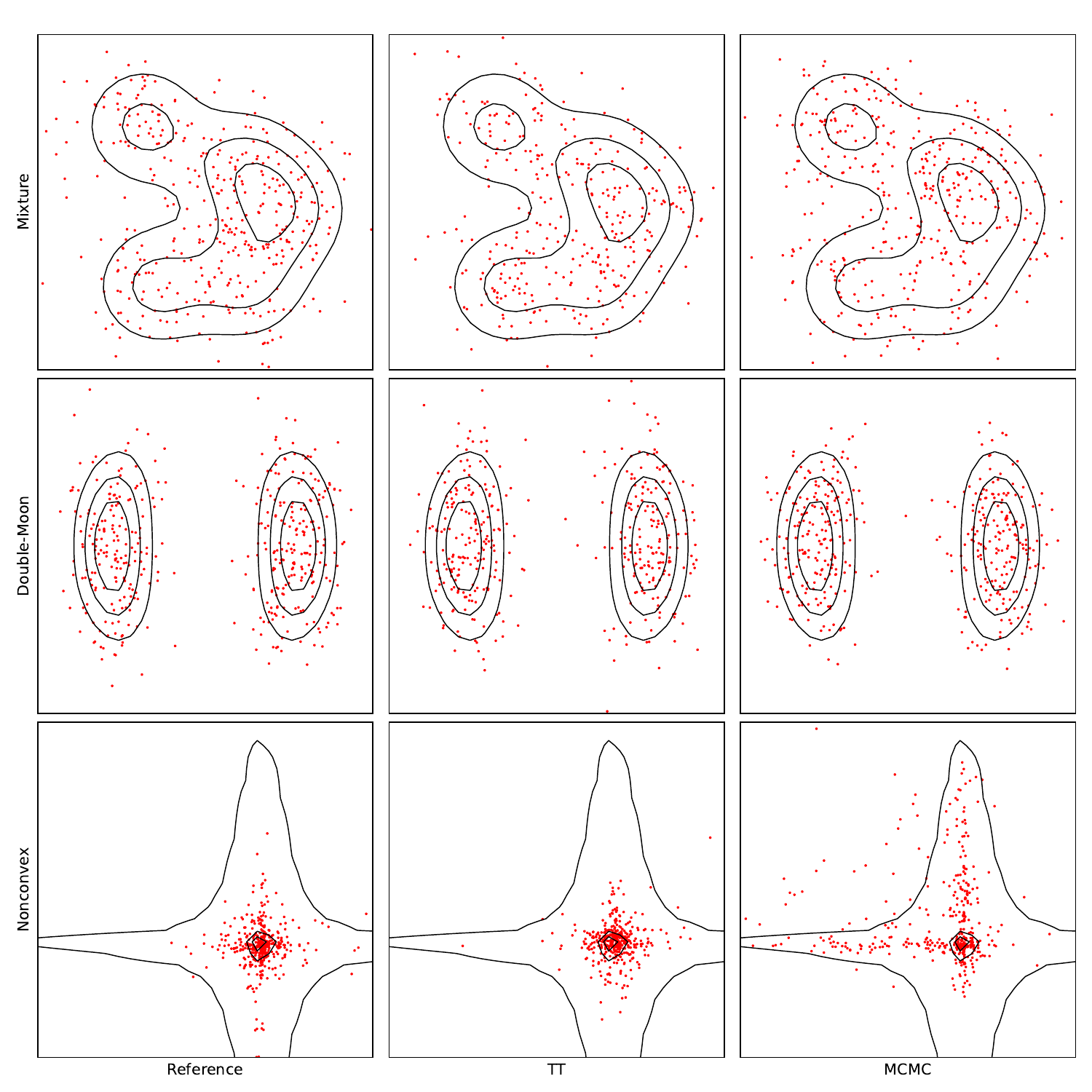}
        \caption{Contour lines of two-dimensional marginals of the test distributions shown with the reference sample, the TT-generated sample, and an MCMC sample using a shorter chain with an equivalent number of posterior calls as the TT method.}
        \label{fig:sec5_plot}
    \end{figure}

    \section{Application to Bayesian inverse problems}
    The present section consists  of case studies for two certain PDE-constrained Bayesian inverse problem. 
    We consider the application of our method for Bayesian inversion.
    We assess our method in the context of estimating the unknown model parameters while providing an uncertainty quantification.
    We also demonstrate how the fitted TT model for the Bayesian posterior can be re-used to efficiently solve an associated problem.
    In our case, importance sampling for some quantity of interest non-trivially depending on the unknown parameters is considered.
    
    Specifically, we aim to reconstruct the initial condition in an initial value problems for PDEs.
    Our two test cases are described in the following.
    \paragraph{Hyperbolic equation}
        The initial value problem for the wave equation is considered.
        It is given by
        \begin{gather}\label{eq:ip_hyperbolic_pde}
            \frac{\partial^2 u}{\partial x^2} - \frac{\partial^2 u}{\partial t^2} = 0, \\
            u(0, x) = h_\theta(x),\ \frac{\partial u}{\partial t}(0, x) = 0,
        \end{gather}
        where the initial condition depends on the unknown parameters $\theta$, the position of <<peaks>> in the initial distribution
        \begin{equation*}
            h_\theta(x) = \sum\limits_{i=1}^d e^{-(x - \theta_i)^2}.
        \end{equation*}
        The solution to this problem is given by
        \begin{equation*}
            u(\theta; t, x) = \frac{1}{2}(h_\theta(x - t) + h_\theta(x + t)).
        \end{equation*}
    \paragraph{Parabolic equation}
        Consider the initial-boundary value problem for the heat equation
        \begin{gather}\label{eq:ip_parabolic_pde}
            \frac{\partial u}{\partial t} - \frac{\partial^2 u}{\partial x^2} = 0, \\
            \left. \frac{\partial u}{\partial x} \right|_{x=\{-1, 1\}} = 0,\\
            u(0, x) = h_\theta(x)
        \end{gather}
        with the initial condition parametrized by a truncated Fourier series in the cosine basis
        \begin{equation*}
            h_\theta(x) = \sum\limits_{j=0}^{d-1} \theta_i \cos{(j \pi x)}.
        \end{equation*}
        The solution to this problem is given by
        \begin{equation*}
            u(\theta; t, x) = \sum\limits_{j=0}^{d-1} \theta_j e^{-(\pi j)^2 t}\cos{(j \pi x)}.
        \end{equation*}
    For both forward models, the inverse problem is formulated in the same way.
    Firs, a Gaussian prior on the parameters is chosen by
    \[
        \theta_k \sim \mathcal{N}(0, \sigma_{0,k}^2) \qquad \text{ independent on each other}.
    \]
    For the hyperbolic problem, these distributions have the same standard deviation $\sigma_{0,k} = \sigma_0$.
    For the parabolic ones, the standard deviation decays like $\sigma_{0,k} = \frac{\sigma_0}{k + 1}$, corresponding to a truncation of a trace-class covariance operator in a Karhunen-Lo\`eve expansion.
    A set of parameters $\theta^*$ from the prior distribution is sampled and fixed as <<true>> parameters.
    Second, noisy measurements of the solution
        \[
            \tilde u_{ij} = u(\theta^*; t_i, x_j) + \xi_{ij},\quad \xi_{ij} \sim \mathcal{N}(0, \sigma_{meas}^2\operatorname{I}_d)\quad \text{ i.i.d.}
        \]
    are generated.
    The space and time discretization points are located equidistantly.
    The goal is to reconstruct the parametrized initial condition given these measurements, their likelihood function and the prior distribution of the parameters.
    The potential of the posterior distribution is defined by
        \begin{equation}\label{eq:post_potential_hyperbolic}
            V_\infty(\theta) = \frac{1}{2\sigma_{meas}^2}\sum\limits_{i,j}|u(\theta; t_i, x_j) - \tilde u_{ij}|^2
                + \frac{1}{2}\sum\limits_{k=1}^d\frac{1}{\sigma_{0,k}^2}|\theta_k|^2.
        \end{equation}
    We aim to sample from the posterior distribution.
    In addition to comparing the sample to the reference, we provide an estimate to the parameters along with some uncertainty quantification.
    Since from the fitted TT approximation a discretized version of the marginal probability density of each of the parameters can be obtained, various statistical quantities such as means, variances, quantiles, etc. can be estimated. 
    We choose to provide a \emph{maximal a posteriori} (MAP) estimate and an $89\%$-highest density credible interval, i.e. the minimal width interval containing $89\%$ of the marginal posterior.

    One can notice that in its original form the forward problem of the hyperbolic equation is invariant under permutation of the parameters (positions of the <<peaks>> in the initial distribution).
    In the preliminary computations we noticed the said invariance leads to a posterior distributions with very many modes, which then hinders the convergence of all the methods and renders the interpretation of the results difficult.
    Thus, we choose to enforce the condition that the parameters are in increasing order by sampling with the perturbed potential
    \[
        V_{\infty,s}(\theta) = V_\infty(\theta) + \begin{cases}
                                            0\text{, if }\theta_1 \leq \theta_2 \leq \dots \leq \theta_d \\
                                            +\infty\text{, otherwise}
                                        \end{cases}.
    \]

    First, the approximation to the posterior with the TT method is constructed.
    Next, samples with the TT model and MCMC with the same amount of the posterior calls are generated and compared in the same fashion as in the previous section. 
    \begin{table}[ht]
        \centering
        \begin{tabular}{lccccccccc}
\toprule
Distribution & $d$ & $r_{max}$ & \multicolumn{2}{c}{$N_\infty$} & \multicolumn{3}{c}{$S^2_{\varepsilon}\times 10^3$} & \multicolumn{2}{c}{Double OT} \\
 &  &  & Unique & Total & Ref. to ref. & Ref. to TT & Ref. to MCMC & TT & MCMC \\
\midrule
Hyperbolic & $6$ & 1 & $22K$ & $92K$ & $0.56 \pm 0.21$ & $5.23 \pm 0.84$ & $277.37 \pm 31.34$ & $2.2 \cdot 10^{-5}$ & $7.8 \cdot 10^{-2}$ \\
Parabolic & $10$ & 1 & $15K$ & $60K$ & $1.61 \pm 0.28$ & $2.09 \pm 0.30$ & $781.28 \pm 55.29$ & $2.4 \cdot 10^{-7}$ & $6.1 \cdot 10^{-1}$ \\
\bottomrule
\end{tabular}

        \smallskip
        \caption{Summary of results for each problem showing its dimension, maximal TT rank, posterior call counts, average Sinkhorn distance with standard deviation for each sampling method, and the double OT estimate between distributions.}
        \label{tab:sec6_result}
    \end{table}
    The results are presented in \tabref{tab:sec6_result}, with the same metrics tracked as in the previous section. 
    One can observe that quite a good approximation superior to the baseline method can be achieved by our TT representation.
    The caching provides at least $4$ times acceleration in terms of the number of posterior calls.

    \begin{table}[ht]
        \begin{subtable}{\textwidth}
            \centering
            \begin{tabular}{c||ccc|ccc|c}
\toprule
$\theta^*$ & \multicolumn{3}{c}{$\theta_{\text{MCMC}}$} & \multicolumn{3}{c}{$\theta_{\text{TT}}$} & $\epsilon_{\text{rel}}$ \\
 & min & MAP & max & min & MAP & max &  \\
\midrule
$-2.30$ & $-2.31$ & $-2.13$ & $-1.90$ & $-2.33$ & $-2.15$ & $-1.97$ & $0.11$ \\
$-1.07$ & $-1.24$ & $-1.15$ & $-0.82$ & $-1.24$ & $-1.06$ & $-0.86$ & $0.04$ \\
$-0.61$ & $-0.84$ & $-0.42$ & $-0.39$ & $-0.81$ & $-0.64$ & $-0.49$ & $0.14$ \\
$-0.53$ & $-0.44$ & $-0.37$ & $-0.04$ & $-0.36$ & $-0.21$ & $-0.07$ & $0.13$ \\
$0.87$ & $0.74$ & $0.96$ & $1.12$ & $0.73$ & $0.88$ & $1.04$ & $0.12$ \\
$1.62$ & $1.69$ & $1.91$ & $2.13$ & $1.70$ & $1.91$ & $2.07$ & $0.08$ \\
\bottomrule
\end{tabular}

            \smallskip
            \caption{Hyperbolic}
            \label{tab:parameters_hyperbolic}
        \end{subtable}
        \vfill
        \begin{subtable}{\textwidth}
            \centering
            \begin{tabular}{c||ccc|ccc|c}
\toprule
$\theta^*$ & \multicolumn{3}{c}{$\theta_{\text{MCMC}}$} & \multicolumn{3}{c}{$\theta_{\text{TT}}$} & $\epsilon_{\text{rel}}$ \\
 & min & MAP & max & min & MAP & max &  \\
\midrule
$1.62$ & $1.60$ & $1.64$ & $1.69$ & $1.58$ & $1.67$ & $1.70$ & $0.17$ \\
$0.31$ & $0.13$ & $0.24$ & $0.34$ & $0.12$ & $0.23$ & $0.33$ & $0.03$ \\
$0.18$ & $0.03$ & $0.24$ & $0.45$ & $0.05$ & $0.25$ & $0.47$ & $0.05$ \\
$0.27$ & $-0.35$ & $-0.24$ & $0.27$ & $-0.36$ & $-0.05$ & $0.24$ & $0.03$ \\
$-0.17$ & $-0.34$ & $-0.01$ & $0.26$ & $-0.35$ & $-0.04$ & $0.25$ & $0.00$ \\
$0.38$ & $-0.28$ & $0.02$ & $0.21$ & $-0.26$ & $-0.01$ & $0.26$ & $0.07$ \\
$-0.25$ & $-0.23$ & $-0.01$ & $0.22$ & $-0.22$ & $0.00$ & $0.23$ & $0.02$ \\
$0.10$ & $-0.19$ & $0.03$ & $0.21$ & $-0.19$ & $0.00$ & $0.20$ & $0.02$ \\
$-0.04$ & $-0.15$ & $-0.02$ & $0.20$ & $-0.17$ & $-0.00$ & $0.17$ & $0.08$ \\
$0.02$ & $-0.16$ & $-0.00$ & $0.15$ & $-0.16$ & $0.00$ & $0.16$ & $0.02$ \\
\bottomrule
\end{tabular}

            \smallskip
            \caption{Parabolic}
            \label{tab:parameters_parabolic}
        \end{subtable}
        \caption{Parameters and their credible intervals, estimated by the TT method and the baseline MCMC. Low relative error $\epsilon_{rel}$ shows good agreement between the methods.}
        \label{tab:ip_result}
    \end{table}
    The results of the parameter estimation is presented in \tabref{tab:ip_result}.
    The leftmost column provides the true parameters $\theta^*$.
    Next, minimal and maximal value of a \emph{Highest Density Interval} (HDI) along with the MAP estimate are presented.
    This is first for the reference MCMC chain with $\gtrsim 10^5$ iterations and then for the TT method.
    The error of determination of the intervals relative to their respective lengths (as uncertainty quantification) is provided in the rightmost column.
    \begin{figure}[ht]
        \begin{subfigure}{0.45\textwidth}
            \centering
            \vfill
            \includegraphics[width=\textwidth]{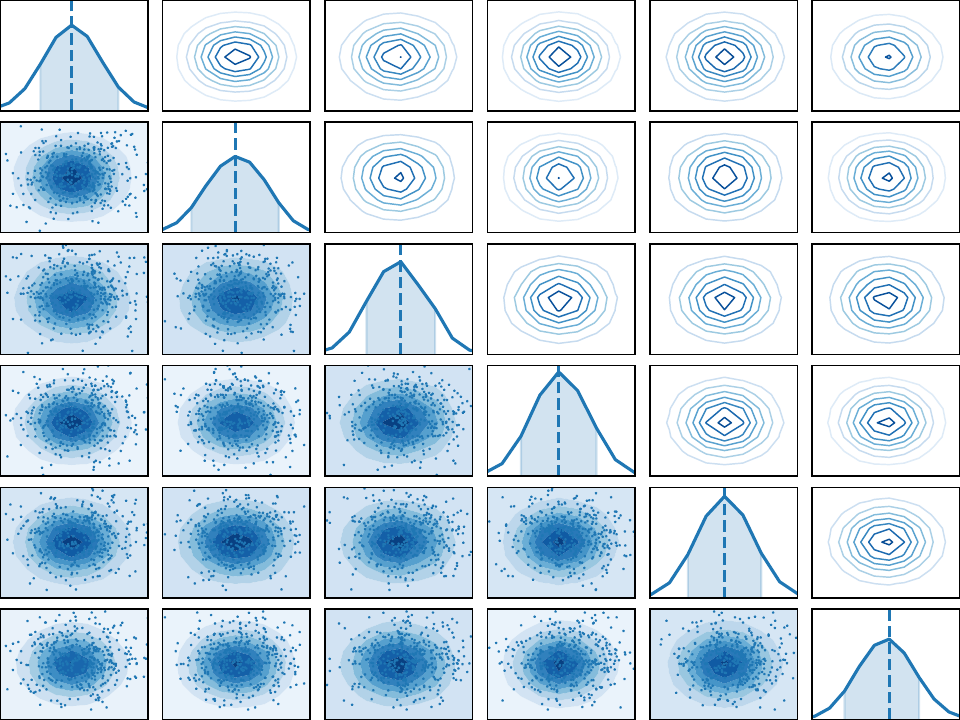}
            \vfill
            \caption{Contour lines and samples of $1$-~and $2$-d marginals of the TT approximation of the posterior. The $89\%$ credible interval for each parameter is filled.}
            \label{fig:hyperbolic_distribution}
        \end{subfigure}
        \hfill
        \begin{subfigure}{0.5\textwidth}
            \centering
            \vfill
            \includegraphics[width=\textwidth]{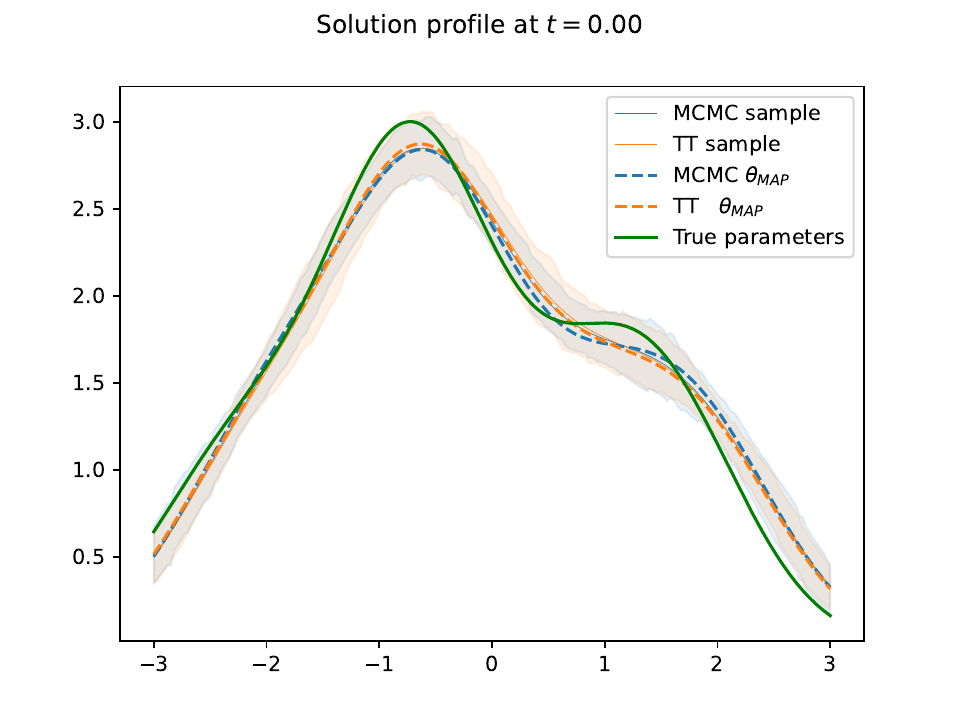}
            \vfill
            \caption{Reconstruction of the initial condition by the TT method and MCMC reference. The filled patch represents a $89\%$ credible interval for the solution $\left. u(x)\right|_{t = 0}$ estimated via samples.}
            \label{fig:hyperbolic_solution}
        \end{subfigure}
        \caption{Posterior distribution and reconstruction of the initial solution for the hyperbolic problem.}
        \label{fig:res_hyperbolic}
    \end{figure}
    \begin{figure}[ht]
        \begin{subfigure}{0.45\textwidth}
            \centering
            \vfill
            \includegraphics[width=\textwidth]{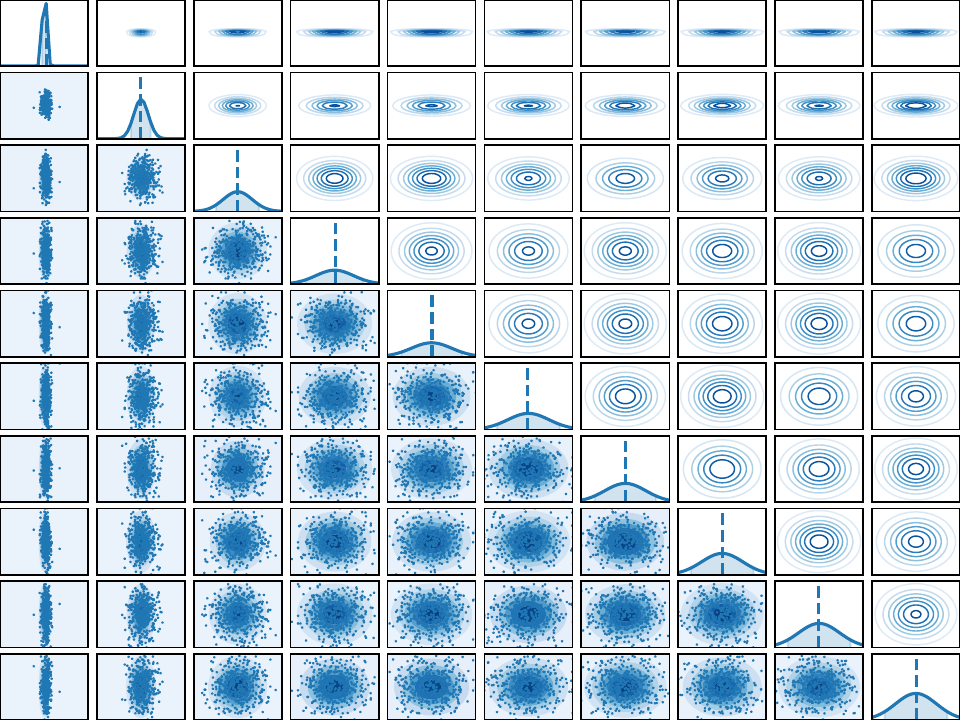}
            \vfill
            \caption{Contour lines and samples of $1$-~and $2$-d marginals of the TT approximation of the posterior. The $89\%$ credible interval for each parameter is filled.}
            \label{fig:parabolic_distribution}
        \end{subfigure}
        \hfill
        \begin{subfigure}{0.5\textwidth}
            \centering
            \vfill
            \includegraphics[width=\textwidth]{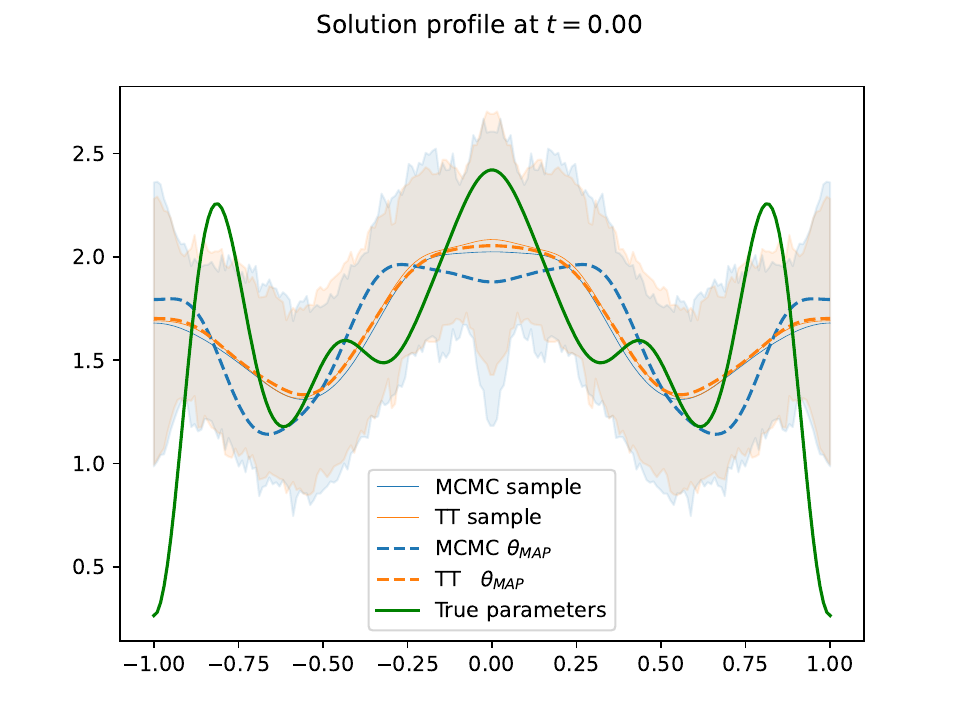}
            \vfill
            \caption{Reconstruction of the initial condition by the TT method and MCMC reference. The filled patch represents a $89\%$ credible interval for the solution $\left. u(x)\right|_{t = 0}$ estimated via samples.}
            \label{fig:parabolic_solution}
        \end{subfigure}
        \caption{Posterior distribution and reconstruction of the initial solution for the parabolic problem.}
        \label{fig:res_parabolic}
    \end{figure}

    \subsection{Modelling of the importance distribution}
        We argue that the structured representation of the posterior density as a discretized function provides the opportunity to modify the solution while reusing the computationally expensive posterior calls.
        For example, new variables can be added to the distribution by concatenating new cores to the Tensor Train representing the posterior density.
        In case when the forward operator calls are the bottleneck in computation of the posterior density, these can be cached (instead of the posterior calls) and reused in posterior fitting for multiple experiments with different measurements but within the same parameter bounds.
        Leaving a thorough examination of these extensions to further work, we provide a single example, namely demonstrating that approximating the importance distribution for a certain quantity of interest can be carried out without any additional calls to the posterior density.

        To briefly summarize the idea of importance sampling, given a certain quantity of interest $F$ for which $\mathbb{E}_{\rho_\infty}[F]$ has to be estimated, one designs a distribution $\rho_F$ and a weighting function $w$ such that $\mathbb{E}_{\rho_F} [wF] = \mathbb{E}_{\rho_\infty}[F]$ and the variance is minimized.
        The \emph{importance weight} takes the form 
        \begin{equation}\label{eq:importance_weight}
            w(x) = \frac{\rho_\infty(x)}{\rho_F(x)}
        \end{equation}
        and in case of the optimal importance distribution~\cite[{Chapter~{5.5}}]{scheichl2021numerical},
        \begin{gather}
            \rho_{F,\text{opt}}(x) \propto |F(x)|\rho_\infty(x) \quad\text{and}\quad
            w_\text{opt}(x) = \frac{1}{|F(x)|}.\label{eq:importance_weight_opt}
        \end{gather}        
        The estimator for $\mathbb{E}_{\rho_\infty}[F]$ is then given by
        \begin{equation}\label{eq:importance_estimator}
            \hat{F}_N = \frac{\sum\limits_{i=1}^N w(x_i) F(x_i)}{\sum\limits_{i=1}^N w(x_i) }\qquad \text{with}\quad x_i \sim \rho_F \text{ i.i.d.}
        \end{equation}
        Although $\rho_\infty$ and $\rho_F$ are both known only up to multiplicative factors, these factors cancel out in~\eqref{eq:importance_estimator}.
        We note that this estimate is biased.

        To demonstrate importance sampling with the TT model, we modify the setting used in the inverse problem for the parabolic equation.
        Mimicking certain real-world settings, we assume that the solution in the inside area is of practical interest but not accessible for direct measurements~\cite{martin1996inverse,martin2024inverse,carasso1992space}.
        We use the same geometry and the same priors on the parameters as in the parabolic problem in the previous section but assume there are no measurement points in $x \in [-\frac{1}{2}, \frac{1}{2}]$ and the $10$ measurements are taken outside of the interval with equal spacing.
        We suppose that a reliable measurement of $\mathbb{E}_{\theta \sim \rho_\infty} u(\theta; 0, 0)$ is required.
        Note that the computation of the function $F(\theta) := u(\theta; 0, 0)$ only requires the evaluation of the parametrized initial condition but not the expensive forward model.
        
        We fit one step of the method with initial step $\rho_\text{TT}$ (the TT approximation of the posterior) and the target~$\propto\|F\|\rho_\text{TT}$, assuming the resulting approximation $\rho_{F,\text{TT}}$ to be an adequate approximation to the importance distribution.
        Then, $10$ samples of size $N_{\text{samples}} = 400$ each are generated.
        And for each of them, the expectation of $F$ is estimated via the estimator~\eqref{eq:importance_estimator}.
        The computation of the true importance weight~\eqref{eq:importance_weight} still requires $N_\text{samples}$ posterior calls.
        We have noted however that the resulting distribution of the estimates for $F$ does not change significantly if the optimal weight~\eqref{eq:importance_weight_opt} is used instead. 
        Thus, the estimate can be computed without any posterior calls at all.
        This approach is compared to estimating the mean directly from the samples $F(x),\ x\sim \rho_\infty$.
        The results are depicted in Figure~\ref{fig:importance_mean}.
        \begin{figure}[ht]
            \centering
            \includegraphics[width=0.5\linewidth]{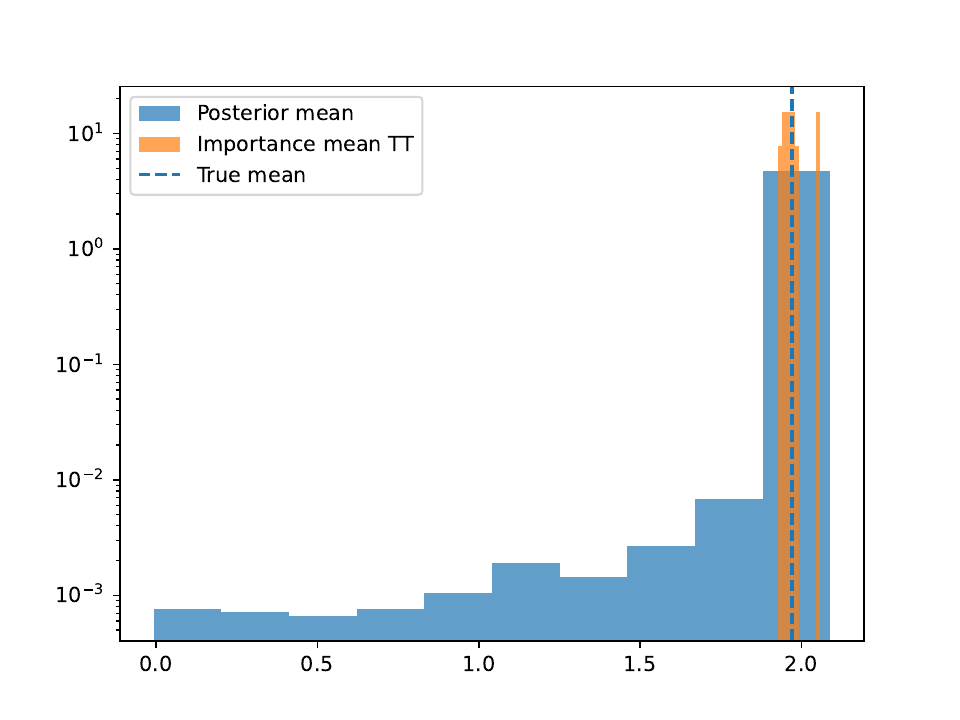}
            \caption{Distribution of the estimate of $\mathbb{E}_{\rho_\infty}F$ via the posterior mean (blue) and the importance estimator~\eqref{eq:importance_estimator} along with the true mean (vertical blue line).}
            \label{fig:importance_mean}
        \end{figure}
        One can see that the estimates without importance sampling (blue) are quite spread out.
        The estimates via the importance distribution are not centered around the actual mean due to the bias in the estimate, but in general are much closer to it.
        Thus, the importance estimator visibly reduces variance of the estimation, and is a useful tool, provided that with our method it can be acquired without any additional posterior calls.

\section{Discussion}
    With the current implementation using a finite difference approximation of the functions, we are already able to demonstrate the main ideas of the proposed new method, namely the advantages of the Eulerian approach, the benefit of caching and the re-usability of the results for similar problems.
    There is however still a lot of room for improvements.
    The interpolation could be improved by using the \emph{functional Tensor-Train} format~\cite{oseledets2013constructive}.
    A careful choice of the basis function should be dictated by the properties of the problem.
    To allow for an even more extensive re-use of posterior calls, one could consider building a model for the joint distribution $\rho(x, y)$ and conditioning it to certain values of the measurement $y$.
    In this approach, the model is re-used to estimate multiple parameter sets from multiple measurements carried out with the same experimental setup.
    Numerically, this can be achieved either by directly conditioning the TT model for $\rho(x, y)$ by means of a multidimensional integration or by considering proximal steps with respect to the conditional Wasserstein distance~\cite{chemseddine2024conditional}.

    The regularized JKO scheme presented in the current paper can in principle be viewed as a general-purpose minimization algorithm in Wasserstein space.
    This can be achieved as long as an efficient computation of the energy first variation of a general energy by~\eqref{eq:fp_general_end} can be implemented.
    Energies different from $\operatorname{KL}$ divergence may for example arise in case when in addition to approximating the target distribution the model has to meet certain constraints such as fairness, safety or interpretability (see~\cite{liu2021sampling} and references therein).

\section*{Acknowledgements}
    VA acknowledges the support of the EMPIR project 20IND04-ATMOC.
    This project (20IND04 ATMOC) has received funding from the EMPIR programme co-financed by the Participating States and from the European Union's Horizon 2020 research and innovation programme.
    ME was supported by ANR-DFG project ``COFNET'' and DFG SPP 2298 ``Theoretical Foundations of Deep Learning''.
    
    The authors would like to express their gratitude to Mathias Oster, David Sommer, Robert Gruhlke and Reinhold Schneider for helpful discussions.

    \texttt{ChatGPT} was used in the process of editing of the present paper.

\bibliographystyle{unsrt}  
\bibliography{references}

\end{document}